\documentclass[a4paper,10pt]{amsart}

\usepackage{hyperref}
\usepackage{enumitem}
\usepackage{xcolor}
\usepackage{comment}
\usepackage{amssymb}
\usepackage{graphicx} 
\usepackage{tikz-cd}
\usepackage{amsmath}
\newcommand{\R}{\mathbb{R}}
\newcommand{\C}{\mathbb{C}}

\newcommand{\Z}{\mathbb{Z}}

\newcommand{\im}{\text{im}}

\newcommand{\del}{\partial}
\newcommand{\dbar}{\overline{\partial}}
\newcommand{\cala}{\mathcal{A}}
\newcommand{\calh}{\mathcal{H}}
\theoremstyle{definition}
\newtheorem{defi}{Definition}
\newtheorem{rmk}{Remark}
\newtheorem{exa}{Example}
\theoremstyle{plain}
\newtheorem{lemma}{Lemma}
\newtheorem{cor}{Corollary}
\newtheorem*{cor*}{Corollary}
\newtheorem{prop}{Proposition}
\newtheorem{thm}{Theorem}
\newtheorem*{thm*}{Theorem}

\newtheorem*{conj*}{Conjecture}
\newtheorem*{question*}{Question}

\title{Geometrically formal Hermitian manifolds}
\newtheorem{question}{Question}
\author{Tommaso Sferruzza}
\address{Dipartimento di Matematica ``Giuseppe Peano"\\
Università di Torino\\
Via Carlo Alberto 10, 10121\\
Torino, Italy}
\email{tommaso.sferruzza@gmail.com}
\author{Adriano Tomassini}
\address{Dipartimento di Scienze Matematiche, Fisiche e Informatiche\\
Unit\`{a} di Matematica e Informatica,
Universit\`{a} degli Studi di Parma\\
Parco Area delle Scienze 53/A, 43124 \\
Parma, Italy}
\email{adriano.tomassini@unipr.it}
\thanks{The first author holds a research fellowship of Istituto Nazionale di Alta Matematica and he is partially supported by the GNSAGA project ``Progetti di Ricerca 2025 - CUP E53C24001950001". The first author is partially supported by the Project PRIN 2022 ``Real and Complex Manifolds: Geometry and Holomorphic Dynamics 2022AP8HZ9'' and both authors are partially supported by GNSAGA of INdAM}
\begin{document}

\begin{abstract}
We study geometrically formal Hermitian metrics on compact complex manifolds, focusing on Dolbeault, Bott-Chern, and Aeppli cohomologies. We establish topological and cohomological obstructions to their existence and we provide a detailed analysis for several classes of compact complex manifolds. We prove that on any blow-up of a K\"ahler manifold no K\"ahler metric can be geometrically formal, and that K\"ahler metrics with nonnegative curvature operator are necessarily geometrically formal. 
\end{abstract}
\maketitle
\tableofcontents

\section{Introduction}
On compact globally symmetric spaces and rational homology spheres, it is well known that the products of harmonic forms are harmonic. Yet, this fails for most closed Riemannian manifolds. 
 
Motivated by the study of rational homotopy theory, Sullivan \cite{Sull73} observed that “there are topological obstructions for M to admit a metric in which the product of harmonic forms is still harmonic”: for instance, the rational formality of the closed oriented manifold and the vanishing of every Massey product.

Later, Kotschick \cite{Kot01} defined \emph{(metrically) formal metrics} (usually called {\em geometrically formal metrics}) on a closed manifold as Riemannian metrics for which all wedge products of harmonic forms are harmonic. 

Harmonic forms have constant pointwise norm with respect to a geometrically formal metric \cite[Lemma 4]{Kot01}. This weak holonomy reduction principle allowed Kotschick \cite{Kot01} to obtain  several topological obstructions for the existence of such metrics (often non-trivial on rationally formal manifolds), and to characterize the cohomology ring of a geometrically formal manifold up to dimension four. For further topological obstructions, see \cite{OrnPil}.

In the wake of Kotschick's foundational paper, a number of subsequent works investigated the existence of geometrically formal metrics on several classes of closed oriented Riemannian manifolds  \cite{Kot1,KT1,KT2,OrnPil,Kasuya-Kokyu,Kasuya-Vaisman-metrics,AZ,B14}.
 
In the holomorphic setting, a compact K\"ahler manifold (and more generally, a compact complex manifold satisfying the $\del\dbar$-lemma) is notably always formal in the sense of Sullivan \cite{DGMS}. Nonetheless, a K\"ahler metric is not always geometrically formal: in \cite{Huy00}, Huybrechts relates the failure of geometric formality for K\"ahler metrics to the geometric properties of the \emph{non-linear K\"ahler cone}, whereas Nagy \cite{Nagy} proved structural theorems and completely described the first and second Betti numbers of geometrically formal  K\"ahler  three-folds.

On the other hand, the harmonicity of products of harmonic forms led Merkulov  \cite{Mer} to study the $\mathcal{A}_{\infty}$-structure on the space of harmonic forms of K\"ahler manifolds. A twisted version of this structure was introduced by Polishchuk to confirm Kontsevich’s
homological version of mirror symmetry on elliptic curves \cite{Pol}.
 
Outside of the $\del\dbar$-lemma realm, \emph{geometric Dolbeault, Bott-Chern, Aeppli, and ABC-geometric formality} have recently been introduced \cite{TomTor14, AT15, MS24,AS20} as Hermitian adaptations of geometrically formal metrics to the harmonic forms with respect to the \emph{Dolbeault, Bott-Chern, and Aeppli harmonic Laplacians} (see Section \ref{sec:not} for the definitions). These metrics arise naturally in \emph{Dolbeault homotopy theory} \cite{NT78} and \emph{pluripotential homotopy theory} \cite{MS24}. In particular, their existence is obstructed by a respective notion of Massey products (see Section \ref{sec:not}).

Several authors have addressed the existence of geometrically formal Hermitian metrics on compact complex manifolds with special metric structures. In  \cite{IO23}, Istrati and Otiman characterized the existence of geometrically Dolbeault formal and geometrically Bott-Chern formal metrics on compact Vaisman manifolds. The existence of geometrically Dolbeault formal and geometrically Bott-Chern formal metrics is less restrictive on the class of (non-Vaisman) LCK manifolds.
For instance, Kasuya \cite{Kasuya-Remarks-Dolbeault} proved that Oeljeklaus-Toma manifolds (which are LCK) are always geometrically Dolbeault formal and Oeljeklaus-Toma manifolds of type $(s,1)$ are geometrically Bott-Chern formal. Cesarino and the second named author \cite{CesTom} studied the existence of geometrically Bott-Chern formal metrics and triple ABC-Massey products on Bigalke-Rollenske manifolds, generalized Nakamura manifolds, and families of complex splitting type meta-Abelian solvmanifolds.

In this paper, we systematically study  geometrically formal Hermitian metrics in the above senses and geometrically formal K\"ahler metrics, on compact complex manifolds. The relations between these notions can be found in Diagram \ref{diag:metrics}.

In the first part of the paper, we provide complex and topological obstructions. We start by proving that, on geometrically Bott-Chern formal manifolds, a $\del\dbar$-lemma for $(p,0)$-forms holds (Proposition \ref{prop:1} and Corollary \ref{cor:ddbarlemma-bc}).

Then, in analogy with \cite{Kot01}, we prove that Dolbeault harmonic forms (respectively, Bott-Chern and Aeppli, respectively, bigraded harmonic forms) have constant pointwise norm with respect to geometrically Dolbeault (respectively, ABC-geometrically and geometrically Aeppli) formal metrics, see Propositions \ref{prop:dolb-const-norm}, \ref{prop:abc-form-const}, and \ref{prop:aep-bot-const-norm}. We introduce the notion of \emph{geometrically Bott-Chern (CN) formal metrics} for geometrically Bott-Chern formal metrics with Bott-Chern harmonic forms of constant pointwise norm. As a consequence, we establish upper and lower bounds on the Hodge numbers, the dimensions of the Bott-Chern and Aeppli cohomology spaces, and, via Fr\"olicher inequalities \cite{Fro55,AT15}, on the Betti numbers of the underlying manifold, see Theorem \ref{thm:geom-dolb-top}, Theorem \ref{thm-abc-geom-top}, Corollary \ref{cor:geom-BC-form-const-norm}, and Theorem \ref{thm:geom-aeppl-top}.  As an immediate application, we obtain the following.
\begin{thm*}[Theorem \ref{thm:blow-up-torus}]
The blow-up of any complex torus along any compact complex submanifold of complex codimension at least 2 is not geometrically formal, nor geometrically Dolbeault
formal, nor geometrically Bott-Chern (CN) formal.
\end{thm*}
We provide further obstructions by studying the Albanese map of geometrically Bott-Chern formal manifolds endowed with Bott-Chern harmonic $(1,0)$-forms of constant pointwise norm.
\begin{thm*}[Theorem \ref{thm:fibr-ABC}]
Let $M$ be a compact complex manifold admitting a geometrically Bott-Chern formal metric such that the Bott-Chern harmonic $(1,0)$-forms have constant pointwise norm. Let $k:=\dim_{\C} \text{Alb} (M)$. Then, there exists a surjective holomorphic submersion $\pi\colon M\rightarrow \mathbb{T}^{k}_{\C}$ so that $\pi^*$ is injective in $H_{BC}^{\bullet,\bullet}$, $H_{\dbar}^{\bullet,0}$, and $H_{\del}^{0,\bullet}$. Moreover, if 
\begin{itemize}
    \item the geometrically Bott-Chern formal metric is K\"ahler, then $k:=b_1(M)/2$ and $M$ admits a holomorphic submersion onto a complex torus that is injective in de Rham cohomology with complex coefficients;
    \item $k=b_1(M)/2=n$, then $M$ is biholomorphic to a complex torus and the metric $g$ is flat.
\end{itemize}
\end{thm*}
In second part of the paper, we focus on compact complex surfaces and blow-ups. In \cite{AT15}, it is proved that compact complex surfaces diffeomorphic to solvmanifolds are geometrically Dolbeault formal, except for primary Kodaira surfaces. Here, we investigate further several classes of (not necessarily locally homogeneous) compact complex surfaces.  
\begin{thm*}[Theorem \ref{thm:rational} and Theorem \ref{thm:ruled}]
Hermitian metrics on the Hirzebruch surfaces $\Sigma_n$ with $n\geq 2$ odd, are never geometrically Aeppli formal.

Ruled surfaces of genus $g\geq  2$, class VII minimal surfaces with $b_2>6$, classical Enriques surfaces, and $K3$ surfaces are not geometrically  formal, nor geometrically Dolbeault formal, nor geometrically Bott-Chern (CN) formal.
\end{thm*}
On blow-ups of compact complex manifolds, in particular, on non-minimal compact complex surfaces, we observe that geometric formality, geometric Dolbeault, geometric Bott-Chern (CN) formality can be obstructed by performing a suitable number of blow-ups, so that the cohomology added by exceptional divisors provides obstructions via \cite[Theorem 6]{Kot01}, Theorem \ref{thm:geom-dolb-top} and Corollary \ref{cor:geom-BC-form-const-norm}. On the other hand, Huybrechts proved that geometric formality fails for K\"ahler metrics on any blow-up of a projective variety \cite{Huy00}. We generalize this result as follows.
\begin{thm*}[Theorem \ref{thm:blow-up-kahler}]
On the blow-up of any compact K\"ahler manifold along any compact complex submanifold of complex codimension strictly greater than one, there exist no geometrically formal K\"ahler metrics.
\end{thm*}
Motivated by these results and by several examples involving geometric Bott-Chern formality (Examples \ref{example:1} and \ref{example:2}, K3 surfaces, complex tori, the blow-ups of $\mathbb{CP}^3$ in \cite{PSZ24}), we are led to ask whether the blow-up operation disrupts geometric formality in any sense.
\begin{question}
Can the blow-up of any compact complex manifold admit any type of  geometrically formal Hermitian metrics?
\end{question}
We recall that a Hermitian structure $(g,J,\omega)$ is said to be a \emph{strong K\"ahler with torsion metric} (briefly, \emph{SKT}), if $\del\dbar\omega=0$, respectively, \emph{balanced}, if $d^*\omega=0$.
Such metrics play a relevant role in the study of metric aspects of non-K\"ahler manifolds.

Indeed, for any given Hermitian manifold $(M,J,g)$  of complex dimension $n$, in view of 
\cite{Ga2} there is a $1$-parameter family of canonical Hermitian 
connections on $M$ which can be distinguished by properties of
their torsion tensor $T$. In particular, there exists a unique
connection $\nabla^B$ satisfying $\nabla^B g=0$, $\nabla^B J=0$
for which $g(X,T(Y,Z))$ is totally skew-symmetric. Then, the resulting
$3$-form can then be identified with $Jd\omega$, where $\omega (\cdot, \cdot)
= g(J\cdot, \cdot)$ is the fundamental $2$-form associated to the
Hermitian structure $(J, g)$. This connection was used by Bismut
in \cite{Bi} to prove a local index formula for the Dolbeault
operator when the manifold is non-K\" ahler. The properties of
such a connection are related to what is called \lq\lq {\em
K\"ahler with torsion geometry}\rq\rq\, and if $Jd \omega$ is closed,
or equivalently if $\omega$ is $\partial \overline \partial$-closed,
then the Hermitian structure $(J, g)$ is strong KT and $g$ is
called a {\em strong K\"ahler with Torsion} metric, shortly \emph{SKT} metric or a {\em{pluriclosed}} metric. The
SKT metrics have also applications in type II string theory
and in 2-dimensional supersymmetric $\sigma$-models \cite{GHR,Str}
and have relations with generalized K\"ahler structures (see for
instance \cite{Gu,Hi2,AG,FT}). 

In the terminology of Michelsohn (see
\cite{Mi}), a Hermitian metric $g$ on a complex manifold $(M,J)$ is said to be {\em balanced} if its the fundamental form $\omega$ is co-closed, namely $d\ast\omega = 0$. When $(M, J )$ is compact, the
existence of a balanced metric on $(M, J)$ is intrinsically characterized in terms of currents
(see \cite[Theorem 4.7]{Mi}). Furthermore, by Alessandrini and Bassanelli (see \cite{AB1}, \cite{AB2}) the class of compact balanced manifolds has the nice property to be invariant under modifications.

The Fino-Vezzoni conjecture \cite{FV} states that a compact complex manifold admitting both SKT and balanced metrics is necessarily K\"ahler.

In \cite{ST24} it was proved that on  a nilmanifold endowed with invariant complex structure of complex dimension 3, there exist SKT metrics if and only if there exist geometrically Bott-Chern formal metrics. Vice versa, the existence of geometric Bott-Chern formal metrics implies the existence of SKT metrics for special families of invariant complex structures on nilmanifolds \cite{ST24}. In line with the Fino-Vezzoni conjecture,  we ask the two complementary questions.
\begin{question}
Let $(M,J)$ be a nilmanifold with invariant complex structure. If $(M,J)$ admits a balanced metric, is it true that $(M,J)$ does not admit geometrically Bott-Chern formal metrics?
\end{question}
\begin{question}
Do geometrically Bott-Chern formal nilmanifolds admit SKT metrics?   
\end{question}
Arroyo and Nicolini \cite{AN}, proved that a SKT nilmanifold is necessarily $2$-step, hence a positive answer to Question 3, would prove that geometrically Bott-Chern formal metrics would exist only on $2$-step nilmanifolds.

Furthermore, we study the interplay of geometric formality (and thus, rational and pluripotential homotopy theory) with nonnegative curvature properties of K\"ahler metrics. Nonnegativity of the Ricci curvature implies that harmonic $1$-forms are parallel, by a Bochner technique. However, this hypothesis is not sufficient, as Huybrechts \cite{Huy00} proved that a Calabi-Yau manifold that is birational but not isomorphic to another Calabi-Yau manifold is never geometrically formal.

Under the most robust hypothesis of a K\"ahler metric with nonnegative curvature operator and exploiting the Gallot-Meyer theorem \cite{GM75}, we prove the following.
\begin{thm*}[Theorem \ref{thm:curv-kahl-non-neg-curv}]
Let $M$ be a compact complex manifold with a K\"ahler metric $g$ of nonnegative curvature operator. Then $g$ is geometrically formal.
\end{thm*}
Compact K\"ahler manifolds with nonnegative curvature operators are total spaces of holomorphic locally isometrically trivial fiber bundles  over compact K\"ahler flat manifolds with, as fibers, products of factors that are either biholomorphic to complex projective spaces, or isometric to compact Hermitian symmetric spaces \cite{CC86}. 
In particular, K\"ahler flat manifolds (Corollary \ref{cor:kahler-flat}) and, hence, K\"ahler solvmanifolds (Corollary \ref{cor:kahler-solv}), are  geometrically formal K\"ahler, thus providing an alternative proof of \cite[Theorem 2.4]{ST24} via curvature arguments. 

In the last sections of this paper, we focus on complex parallelisable manifolds, rank 2 even dimensional compact Lie groups with isotropic Samelson complex structures, meta-Abelian solvmanifolds, and Calabi-Eckmann manifolds.

Complex parallelisable manifolds are compact quotients of simply connected complex Lie groups by discrete subgroups \cite{Wan54}. Following the notation of Nakamura, we call a complex parallelisable manifold a \emph{complex solvable (respectively, nilpotent) manifold} if the universal cover is a solvable (respectively, nilpotent) Lie group.

In line with the examples in \cite{CFGU}, we confirm the following.
\begin{thm*}[Theorem \ref{thm:dol:cnm}]
Complex nilpotent manifolds are not Dolbeault formal, unless they are complex tori.
\end{thm*}
This immediately implies that complex nilpotent manifolds are not geometrically Dolbeault formal (Corollary \ref{cor:no-geom-dolb-cn}). It is worth pointing out that the first named author and Verbitsky \cite{SV} proved recently that, more generally, any nilmanifold endowed with an invariant complex structure is Dolbeault formal if and only if it is a complex torus. The situation, however, differs for complex solvable manifolds, as there exist examples admitting geometrically Dolbeault formal metrics (Remark \ref{rmk:geom-dolb-cs}).

Since complex parallelisable manifolds never satisfy the $\del\dbar$-lemma for $(p,0)$-forms unless they are tori, Proposition \ref{prop:1} implies the following.
\begin{thm*}[Theorem \ref{thm:geom-bott-cpm}]
Complex parallelisable manifolds are not geometrically Bott-Chern formal, unless they are complex tori.
\end{thm*}
We push this further and show in Proposition \ref{prop:ABC-Massey-cs} that on complex solvable manifolds up to complex dimension $5$, we can construct non-vanishing triple ABC-Massey products (see Section \ref{sec:not} for their definition). The proof relies heavily on the explicit structure equations for this class of manifolds (see Appendix \ref{appendix}), so that it would be interesting to provide an answer to the natural
\begin{question}
Does every complex solvable manifold admit non-vanishing triple ABC-Massey products?
\end{question}
In Section \ref{sec:7}, by means of computations with the Fr\"olicher spectral sequence, we prove that simple even-dimensional compact Lie groups of rank $2$, endowed with an isotropic Samelson complex structure are Dolbeault formal, see Corollary \ref{cor:7}.

In Section \ref{sec:8}, we give a further characterization for the existence of geometrically Aeppli formal metrics, see Lemma \ref{lemma:hodge-symm-geom-aeppl}, and we prove that, for a meta-Abelian complex splitting type solvmanifold, the property of satisfying the $\del\dbar$-lemma is equivalent to the existence of a geometrically Aeppli formal metric, see Theorem \ref{thm:geom-aeppli-metaabelian}. We apply this Theorem and \cite[Condition 3.3]{Kasuya-Hodge-sym-dec} to prove that generalized Nakamura manifolds satisfying a non-resonance
condition on the lattice \cite{CT} are geometrically Aeppli formal.

In the final section of the present paper, we study Calabi-Eckmann manifolds $M_{u,v}$ (see Section \ref{sec:CE} for their definition). We first provide a multiplicative model for computing Bott-Chern and Aeppli cohomology of any Calabi-Eckmann manifold, thus completing the picture by Stelzig \cite{Ste21} also for the case $u=v$ (Lemmata \ref{lem:model-BC} and \ref{lemma:BC-numbers-CE}).  Then, we conclude by characterizing which standard metrics are geometrically Dolbeault formal and Bott-Chern formal  on Calabi-Eckmann manifolds.
\begin{cor*}[Corollary  \ref{cor:no-geom-BC-std-CE} and Corollary \ref{lemma:dolb-geom-CE-std}]
The standard metric on a Calabi-Eckmann manifold $M_{u,v}$ is geometrically Bott-Chern formal if and only if its diffeomorphism type is either $S^1\times S^1$, $S^1\times S^3$, or $S^3\times S^3$.\\
The standard metric on a Calabi-Eckmann manifold $M_{u,v}$ is geometrically Dolbeault formal if and only if its diffeomorphism type is $S^1\times S^{2v+1}$.
\end{cor*}
Note that the standard metric on $M_{0,v}$ is both geometrically Dolbeault formal and locally conformally K\"ahler \cite{KO}.  On the other hand, the standard metric on $M_{0,0}$, $M_{0,1}$, and $M_{1,1}$ is both geometrically Bott-Chern formal and Bismut-flat \cite{WYZ}. This latter observation suggests a deeper relationship between canonical Hermitian connections and the existence of geometrically formal Hermitian metrics, which we plan to investigate in future work.

As pointed by one of the anonymous referees, it would be interesting to extend the study of geometrically Dolbeault formal and geometrically Bott-Chern formal metrics on compact complex manifolds which admit a trasverse K\"ahler foliation, e.g., LVM manifolds \cite{Meer}, via the techniques introduced by Ishida and Kasuya \cite{IK}.\\

\noindent {\bf Acknowledgements.}
The authors would like to thank Dieter Kotschick for correcting an inaccuracy in a former version of the proof of Theorem 7, Jorge Lauret, Giovanni Placini, and Jonas Stelzig for many useful suggestions and comments that helped improve the presentation of this work, and Sergei Merkulov and Dennis Sullivan for their interest in the paper. The first author would like to thank Jorge Lauret for his kind hospitality during his visiting at FAMAF, Universidad Nacional de Córdoba, Argentina.

\section{Notations and preliminaries}\label{sec:not}
Let $(M,J,g)$ be a compact Hermitian manifold of complex dimension $n$ and let $\ast\colon \mathcal{A}^{p,q}(M)\rightarrow\mathcal{A}^{n-p,n-q}(M)$ be the $\C$-antilinear Hodge $\ast$ operator with respect to $g$, i.e., for every $\alpha,\beta\in \mathcal{A}^{p,q}(M)$, we define
\begin{equation*}
\alpha\wedge \ast\beta=g(\alpha,\beta)\frac{\omega^{n}}{n!},
\end{equation*}
where $g$ is the $\C$-antilinear extension of $g$ to $\mathcal{A}^{\bullet,\bullet}(M)$ and $\omega(X,Y)=g(JX,Y)$ for $X,Y\in TM$ is the fundamental form of $g$, $\omega\in\mathcal{A}^{1,1}(M)\cap \mathcal{A}_{\R}^2(M)$. Sometimes we may denote compact complex manifolds by $M$, without specifying the complex structure $J$.

The \emph{Dolbeault Laplacian,} \emph{Bott-Chern  Laplacian}, and \emph{Aeppli Laplacian} are defined as, respectively, the self-adjoint elliptic operators
\begin{align*}
\Delta_{\dbar}&:=\dbar\,\dbar^*+\dbar^*\dbar,\\
\Delta_{BC}&:=\del\dbar\,\dbar^*\del^*+\dbar^*\del^*\del\dbar+\dbar^*\del\del^*\dbar+\del^*\dbar\,\dbar^*\del+\dbar^*\dbar+\del^*\del,\\
\Delta_A&:=\del\del^*+\dbar\,\dbar^*+\dbar^*\del^*\del\dbar+\del\dbar\,\dbar^*\del^*+\del\dbar^*\dbar\del^*+\dbar\del^*\del\dbar^*,
\end{align*}
and the spaces of \emph{Dolbeault harmonic forms}, \emph{Bott-Chern harmonic forms}, and \emph{Aeppli harmonic forms} as
\begin{align*}
&\mathcal{H}_{\dbar}^{\bullet,\bullet}(M,g):=\ker\Delta_{\dbar}\cap \mathcal{A}^{\bullet,\bullet}(M)=\{\alpha\in \mathcal{A}^{\bullet,\bullet}(M): \dbar\alpha=\dbar^*\alpha=0\} \\
&\mathcal{H}_{BC}^{\bullet,\bullet}(M,g):=\ker\Delta_{BC}\cap \mathcal{A}^{\bullet,\bullet}(M)=\{\alpha\in \mathcal{A}^{\bullet,\bullet}(M): \del\alpha=\dbar\alpha=\del\dbar\ast\alpha=0\}\\
&\mathcal{H}_{A}^{\bullet,\bullet}(M,g):=\ker\Delta_{A}\cap \mathcal{A}^{\bullet,\bullet}(M)=\{\alpha\in \mathcal{A}^{\bullet,\bullet}(M): \del\dbar\alpha=\del\ast\alpha=\dbar\ast\alpha=0\}.
\end{align*}
We note that $\mathcal{H}_{\dbar}^{\bullet,\bullet}(M)$ is $\ast$-invariant and $\alpha\in\calh_{BC}^{p,q}(M)$ if and only if $\ast\alpha\in \calh_A^{n-p,n-q}(M)$, i.e., the map
\[
\ast\colon \calh_{BC}^{p,q}(M)\rightarrow \calh_A^{n-p,n-q}(M)
\]
is an isomorphism for every $p,q$. On the other hand, $\mathcal{H}_{BC}^{\bullet,\bullet}(M)$ and $\calh_{A}^{\bullet,\bullet}(M)$ are invariant under complex conjugation,  whereas $\overline{\calh_{\dbar}^{p,q}(M)}=\mathcal{H}_{\del}^{q,p}(M)$, where the space of \emph{conjugate-Dolbeault harmonic forms} $\mathcal{H}_{\del}^{\bullet,\bullet}(M)$ is the kernel of the conjugate-Dolbeault Laplacian
\[
\Delta_\del:=\del\del^*+\del^*\del.
\]

Thanks to Hodge theory (see e.g., \cite{GriHar,Big69,Big70,Sch01}) adapted to the complexes
\[
\dots\xrightarrow{\dbar}\mathcal{A}^{\bullet,\bullet-1}(M)\xrightarrow{\dbar}\mathcal{A}^{\bullet,\bullet}(M)\xrightarrow{\dbar}\mathcal{A}^{\bullet+1,\bullet+1}(M)\xrightarrow{\dbar}\dots,
\]
and
\[
\mathcal{A}^{\bullet-1,\bullet-1}(M)\xrightarrow{\del\dbar}\mathcal{A}^{\bullet,\bullet}(M)\xrightarrow{\del\oplus\dbar}\mathcal{A}^{\bullet+1,\bullet}(M)\oplus\cala^{\bullet,\bullet+1}(M),
\]
we have the orthogonal decompositions
\begin{align}\label{eq:decomp_hodge:_dolb}
\mathcal{A}^{p,q}(M)&=\calh_{\dbar}^{p,q}(M)\oplus^\perp\im(\dbar|_{\cala^{p,q-1}(M)})\oplus^\perp  \im((\dbar)^*|_{\cala^{p,q+1}(M)})\\
\label{eq:decomp_hodge:_bott}
\mathcal{A}^{p,q}(M)&=\calh_{BC}^{p,q}(M)\oplus^\perp \im(\del\dbar|_{\cala^{p-1,q-1}(M)})\\
&\qquad \oplus^\perp \im(\del^*|_{\cala^{p+1,q}(M)}+\dbar^*|_{\cala^{p,q+1}(M)})\nonumber\\
\label{eq:decomp_hodge:aepl}
\mathcal{A}^{p,q}(M)&=\calh_{A}^{p,q}(M)\oplus^\perp\im(\del|_{\cala^{p-1,q}(M)}+\dbar|_{\cala^{p,q-1}(M)})\\
&\qquad \oplus^\perp  \im((\del\dbar)^*|_{\cala^{p+1,q+1}(M)})\nonumber
\end{align}
and the projections
\begin{gather*}
\calh^{\bullet,\bullet}_{\dbar}(M)\rightarrow H_{\dbar}^{\bullet,\bullet}(M), \qquad \qquad \calh^{\bullet,\bullet}_{BC}(M)\rightarrow H_{BC}^{\bullet,\bullet}(M),\\
\calh^{\bullet,\bullet}_{A}(M)\rightarrow H_{A}^{\bullet,\bullet}(M),
\end{gather*}
which are isomorphisms of complex vector spaces. In general, these maps do not preserve the structure of algebra of Dolbeault and Bott-Chern cohomology spaces induced by the cup product, nor the structure of $H_{BC}^{\bullet,\bullet}(M)$-module of Aeppli cohomology spaces. As a result, the following definitions are well motivated.
\begin{defi}[\cite{TomTor14,AT15,MS24}]
A Hermitian metric $g$ on a compact complex manifold $(M,J)$ is said
\begin{itemize}
    \item \emph{geometrically Dolbeault formal} if $(\calh_{\dbar}^{\bullet,\bullet}(M,g),\wedge)$ is an algebra.
    \item \emph{geometrically Bott-Chern formal} if $(\calh_{BC}^{\bullet,\bullet}(M,g),\wedge)$ is an algebra.
    \item \emph{ABC-geometrically formal} if $\calh^{\bullet, \bullet}:=\calh_{A}^{\bullet,\bullet}(M,g)+\calh_{BC}^{\bullet,\bullet}(M,g)$ is closed with respect to $\del$, $\dbar$, and $\wedge$, or equivalently,  if it is BBA subalgebra of $\mathcal{A}^{\bullet,\bullet}(M)$ closed under $\ast$, satisfies $\del\dbar\equiv 0$, and the inclusion map $\mathcal{H}^{\bullet,\bullet}\hookrightarrow \mathcal{A}^{\bullet,\bullet}(M)$ induces an isomorphism in both Bott-Chern and Aeppli cohomology.
\end{itemize}
A compact Hermitian manifold $(M,J,g)$ is said \emph{Dolbeault geometrically formal}, respectively \emph{Bott-Chern formal, ABC-geometrically formal}, accordingly.
\end{defi}
In view of \cite{AS20}, we define a Hermitian metric $g$ on a compact complex manifold $(M,J)$ to be \emph{geometrically Aeppli formal} if the multiplication of any Aeppli harmonic form $\alpha$ by a Bott-Chern harmonic $\beta$ form is Aeppli harmonic. However, for $\alpha:=1\in\mathcal{H}_{A}^{0,0}(M)$ and any $\beta\in\mathcal{H}_{BC}^{\bullet,\bullet}(M)$, this forces $\mathcal{H}_{BC}^{\bullet,\bullet}(M)=\mathcal{H}_A^{\bullet,\bullet}(M)$. This implies the following.
\begin{lemma}\label{lemma:geom-aeppli-form-def}
Let $(M,J,g)$ be a compact Hermitian manifold. Then the following are equivalent:
\begin{enumerate}
    \item $(M,J,g)$ is geometrically Aeppli formal,
    \item $(M,J,g)$ is ABC-geometrically formal and satisfies the $\del\dbar$-lemma,
    \item the spaces $\mathcal{H}^{\bullet,\bullet}(M):=\mathcal{H}_{\dbar}^{\bullet,\bullet}(M)=\mathcal{H}_{BC}^{\bullet,\bullet}(M)=\mathcal{H}_{A}^{\bullet,\bullet}(M)$ coincide, the space of de Rham harmonic forms decompose in  bigraded forms as $\mathcal{H}_{dR}^{k}(M;\C)=\bigoplus_{p+q=k}\mathcal{H}^{p,q}(M)$ and is closed with respect to $\wedge$.
\end{enumerate}
\begin{proof}
(1) $\implies$ (2). If $(M,J,g)$ is geometrically Aeppli formal, then $\mathcal{H}_{BC}^{\bullet,\bullet}=\mathcal{H}_A^{\bullet,\bullet}$. In particular, the map $H_{BC}\rightarrow H_A$ is an isomorphism and $(M,J)$ satisfies the $\del\dbar$-lemma. The space $\mathcal{H}_A+\mathcal{H}_{BC}$ coincides with $\mathcal{H}_{BC}$, it is closed under $\del$, $\dbar$, and $\wedge$. Hence $g$ is ABC-geometrically formal.

(2) $\implies$ (3). If $g$ is ABC-geometrically formal and satisfies the $\del\dbar$-lemma, then all Aeppli harmonic forms are $d$-closed and, since $H_A\cong H_{BC}$, we have $\mathcal{H}_A=\mathcal{H}_{BC}$. In particular, every Bott-Chern harmonic $(p,q)$-form is $\del$-closed, $\dbar$-closed, $\del^*$-closed, and $\dbar^*$-closed. Since $H_{dR}^{\bullet}\cong \bigoplus H_{BC}^{\bullet,\bullet}$, it holds that $\mathcal{H}^{k}_{dR,\C}=\bigoplus_{p+q=k}\mathcal{H}_{BC}^{p,q}$, and since $\mathcal{H}_{BC}$ is closed with respect to $\wedge$, so is $\mathcal{H}_{dR,\C}$. Since $H_{\dbar}\cong H_{BC}$, it also holds that $\mathcal{H}_{\dbar}=\mathcal{H}_{BC}$. 

(3) $\implies$ (1). By assumption, $\mathcal{H}_A=\mathcal{H}_{BC}$, and it has a structure of algebra with respect to $\wedge$. Hence, the multiplication of Bott-Chern harmonic forms and Aeppli-harmonic forms is Aeppli harmonic.
\end{proof}
\end{lemma}

\begin{defi}[\cite{AT15}]
The \emph{triple ABC-Massey product} of three Bott-Chern cohomology classes $[\alpha]\in H_{BC}^{p,q}(M)$, $[\beta]\in H_{BC}^{r,s}(M)$, and $[\gamma]\in H_{BC}^{u,v}(M)$ satisfying $\del\dbar f_{\alpha\beta}=(-1)^{p+q}\alpha\wedge\beta$ and $\del\dbar f_{\beta\gamma}=(-1)^{r+s}\beta\wedge\gamma$ for some $f_{\alpha\beta}\in \cala^{p+r-1,q+s-1}(M)$, $f_{\beta\gamma}\in\cala^{r+u-1,s+v-1}(M)$ is given by the class
\begin{align*}
\langle[\alpha],[\beta],[\gamma]\rangle_{ABC}&:=[(-1)^{p+q}\alpha\wedge f_{\beta\gamma}-(-1)^{r+s}f_{\alpha\beta}\wedge \gamma]\\
&\in \frac{H_A^{p+r+u-1,q+s+v-1}(M)}{[\alpha]\cup H_A^{r+u-1,s+v-1}(M)+H_A^{p+r-1,q+s-1}(M)\cup[\gamma]}.
\end{align*}
\end{defi}
\begin{defi}[\cite{NT78,MS24}]
A compact complex manifold $(M,J)$ is
\begin{itemize}
    \item \emph{Dolbeault formal} if $(\mathcal{A}^{\bullet,\bullet}(M),\dbar)$ is equivalent, as a differential bigraded algebra, to a differential bigraded algebra $(\mathcal{B},\dbar_{\mathcal{B}})$ such that $\dbar_{\mathcal
    B}\equiv 0$,
    \item \emph{weakly formal} (respectively, \emph{strongly formal}) if $(\mathcal{A}^{\bullet,\bullet},\del,\dbar)$ is equivalent, as a bidifferential bigraded algebra, to a bidifferential bigraded algebra $(\mathcal{B},\del_{\mathcal{B}},\dbar_{\mathcal{B}})$ such that $\del_{\mathcal{B}}\dbar_{\mathcal{B}}\equiv 0$, respectively, $\del_{\mathcal{B}}\equiv 0$, $\dbar_{\mathcal{B}}\equiv 0$.
\end{itemize}
\end{defi}
The relations between the  geometrically formal Hermitian metrics and formality notions can be summarized as follows:
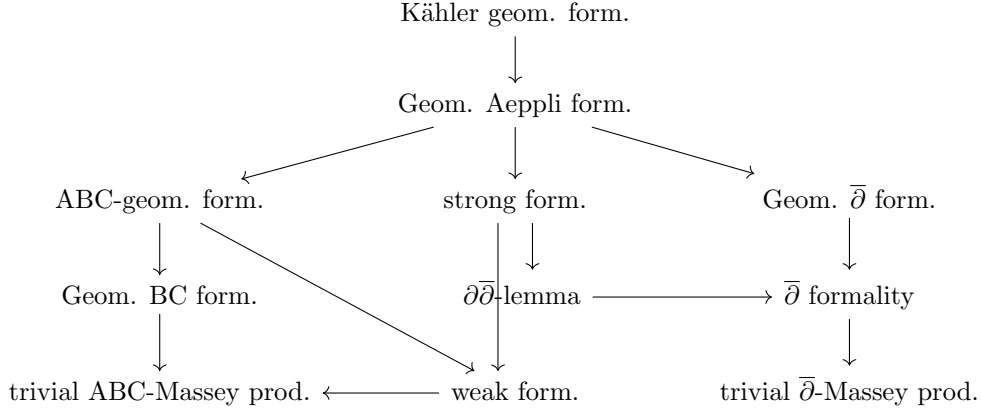
\begin{equation}\label{diag:metrics}
    \begin{tikzcd}
        & \text{ K\"ahler geom. form. }\ar[d]  &\\
        &\text{Geom. Aeppli form.}\ar[dl]\ar[d]\ar[dr]&\\
        \text{ABC-geom. form.}\ar[d]\ar[ddr] &\text{strong form.}\ar[dd, shift right=1.50ex]\ar[d, shift left=1.50ex] & \text{Geom. $\dbar$ form.}\ar[d]\\
        \text{Geom. BC form.}\ar[d] &\,\,\,  \text{$\del\dbar$-lemma}\ar[r]&\text{$\dbar$ formality}\ar[d]&&\\
        \text{trivial ABC-Massey prod.}& \text{weak form.}\ar[l]& \text{trivial $\dbar$-Massey prod.}
    \end{tikzcd}
\end{equation}

\begin{enumerate}[label=(\roman*)]
    \item $\del\dbar$-lemma implies Dolbeault formality \cite{NT78},
    \item geometric Dolbeault formality implies Dolbeault formality \cite{TomTor14},
    \item  geometric Bott-Chern formality, as well as weak formality, implies the vanishing of the triple ABC-Massey products \cite{AT15,MS24},
    \item strong formality implies weak formality and $\del\dbar$-lemma \cite{MS24},
    \item $\del\dbar$-lemma implies that  weak formality and strong formality are equivalent \cite{MS24},
    \item ABC-geometric formality implies geometric Bott-Chern formality and weak formality \cite{MS24},
    \item geometric Aeppli formality implies strong formality, geometric formality and formality in every sense. Moreover, every classical, Dolbeault, and ABC-Massey product vanishes,
    \item K\"ahler geometric formality implies geometric Aeppli formality.
    \end{enumerate}
We prove the first statement of (vii).
\begin{prop}\label{prop:0}
Let $M$ be a compact complex manifold admitting a geometrically Aeppli formal metric $g$. Then $M$ is strongly formal.
\begin{proof}
By Lemma \ref{lemma:geom-aeppli-form-def},  $g$ is ABC-geometrically formal and $(M,J)$ satisfies the $\del\dbar$-lemma. Hence, it is also strongly formal by (iv) and (v).
\end{proof}
\end{prop}
We observe the following facts.
\begin{prop}\label{lemma:change-metric-dolb}
Let $(M,J,g)$ be a geometrically Dolbeault formal manifold  and let $\alpha\in\mathcal{H}_{\dbar}^{1,1}(M)$ be  such that $\alpha^k\neq 0$. For every $f\in\mathcal{A}^{1,0}(M)$, the form $\alpha':=\alpha+\dbar f$, satisfies $(\alpha')^k\neq 0$.
\begin{proof}
If $g$ is geometrically Dolbeault formal, $0\neq \alpha^k\in\mathcal{H}_{\dbar}^{k,k}(M)$. It is easy to see that $(\alpha')^k=\alpha^k+\dbar\gamma$. By the direct sum decomposition \eqref{eq:decomp_hodge:_dolb} for $(k,k)$-forms, it holds that $(\alpha')^k\neq0$. 
\end{proof}
\end{prop}
\begin{prop}\label{lemma:change-metric-bott}
Let $(M,J,g)$ be a geometrically Bott-Chern formal manifold  and let $\alpha\in\mathcal{H}_{BC}^{1,1}(M)$ be  such that $\alpha^k\neq 0$. For every $f\in\mathcal{C}^{\infty}(M)$, the form $\alpha':=\alpha+\del\dbar f$, satisfies $(\alpha')^k\neq 0$.
\begin{proof}
Analogous to the proof of Proposition \ref{lemma:change-metric-dolb}.
\end{proof}
\end{prop}
\section{Complex and topological obstructions}\label{sec:obstr}
In this section, we prove several obstructions related to the cohomology and topology of compact complex manifolds admitting geometrically Dolbeault, Bott-Chern, and Aeppli formal metrics.

We start by observing the following fact, which implies a $\del\dbar$-lemma at the level of $(p,0)$-forms.
\begin{prop}\label{prop:1}
Let $(M,J,g)$ be a $n$-dimensional geometrically Bott-Chern formal manifold. Then, every holomorphic $p$-form is $d$-closed, $p\in\{0,1,\dots, n\}$. 
\begin{proof}
We start by noting that for $\alpha\in\cala^{n,0}(M)$, $d\alpha=\dbar\alpha$, hence, for a $(n,0)$-form $\alpha$, $d\alpha=0$ if and only if $\dbar\alpha=0$.

Fix now $p<n$ and let $\alpha$ be any holomorphic $p$-form, i.e., $\alpha\in\cala^{p,0}(M)$ such that $\dbar\alpha=0$. 

Suppose by contradiction that $\del\alpha\neq0$. Then, $d(\del\alpha)=0$ and by bidegree reasons $(\del\dbar)^*\del\alpha=0$, implying that $\del\alpha\in\calh^{p+1,0}_{BC}(M)$ and, by conjugation, $\dbar\overline{\alpha}\in\calh^{0,p+1}_{BC}(M)$. Since $g$ is geometrically formal, 
$$0\neq \del\alpha\wedge \dbar\overline{\alpha}\in\calh_{BC}^{p+1,p+1}(M).$$ 
However, since 
$$\del\dbar((-1)^{|\alpha|}\alpha\wedge\overline{\alpha})=\del\alpha\wedge\dbar\overline{\alpha},$$ it follows that
\[
0\neq \del\alpha\wedge \dbar\overline{\alpha}\in \calh_{BC}^{p+1,p+1}(M)\cap \im\del\dbar\cap\cala^{p+1,p+1}(M),
\]
which is a contradiction, by the  direct sum decomposition \eqref{eq:decomp_hodge:_bott} of $\cala^{p+1,p+1}(M)$.
\end{proof}
\end{prop}
\begin{cor}\label{cor:ddbarlemma-bc}
On a geometrically Bott-Chern formal manifold $(M,J,g)$ it holds that
\begin{align*} 
\calh_{BC}^{\bullet,0}(M)=\calh_{\dbar}^{\bullet,0}(M),\qquad 
\calh_{BC}^{0,\bullet}(M)=\calh_{\del}^{0,\bullet}(M),\\
\calh_A^{\bullet,n}(M)=\calh_{\dbar}^{\bullet,n}(M),\qquad  \calh_A^{n,\bullet}(M)=\calh_{\del}^{n,\bullet}(M).
\end{align*}
\begin{proof}
We look at the first equality, as the others follow by either complex conjugation, applying $\ast$, or both.

The inclusion $\calh_{BC}^{\bullet,0}(M)\subseteq\calh_{\dbar}^{\bullet,0}(M)$ is always true, since
$$\calh_{BC}^{\bullet,0}(M)=\ker d\cap \cala^{\bullet,0}(M)\subseteq \ker\dbar\cap \cala^{\bullet,0}(M)=\calh_{\dbar}^{\bullet,0}(M).
$$
The opposite inclusion is the thesis of Proposition \ref{prop:1}.
\end{proof}
\end{cor}
On geometrically formal differentiable manifolds, Kotschick \cite[Lemma 4]{Kot01} has proved that the inner product of any two harmonic $k$-forms and the pointwise norm of any harmonic form are constant.  With similar arguments, we can prove the analogous facts.

\begin{prop}\label{prop:dolb-const-norm}
On a geometrically Dolbeault formal manifold $(M,J,g)$, the inner product of any two Dolbeault (respectively, conjugate-Dolbeault) harmonic forms is constant and, in particular, every Dolbeault (respectively, conjugate-Dolbeault) harmonic form has constant pointwise norm.
\begin{proof}
Let $\alpha$ be a Dolbeault harmonic form. Then, $\ast\alpha$ is Dolbeault harmonic and, since $g$ is geometrically Dolbeault formal, $\alpha\wedge\ast\alpha=|\alpha|^2\frac{\omega^n}{n!}\in\calh_{\dbar}^{n,n}(M)$. By definition of Dolbeault harmonic forms,  $\dbar(|\alpha|^2)=0$. Since $M$ is compact, $|\alpha|^2$ is constant. The statement on the inner product follows by polarization and the proof for conjugate-Dolbeault harmonic forms by complex conjugation.
\end{proof}
\end{prop}

\begin{prop}\label{prop:abc-form-const}
On a ABC-geometrically formal manifold $(M,J,g)$, the inner product of any two Bott-Chern (respectively, Aeppli) harmonic forms is constant and, in particular, every Bott-Chern (respectively, Aeppli) harmonic form has constant pointwise norm.
\begin{proof}
It suffices to observe that $\alpha\in\mathcal{H}_{BC}^{p,q}(M)$, then $\ast\alpha\in\mathcal{H}_A^{n-p,n-q}(M)$. Since $g$ is ABC-geometrically formal, $\alpha\wedge \ast\alpha=||\alpha||^2\frac{\omega^n}{n!}\in\mathcal{H}_A^{n,n}(M)+\mathcal{H}_{BC}^{n,n}(M)$. In particular, $\del\dbar\ast(||\alpha||^2\frac{\omega^n}{n!})=0$, i.e., $\del\dbar||\alpha||^2=0$, and $||\alpha||^2$ is pluriharmonic. Since $M$ is compact, however, $||\alpha||^2$ must be constant. 
\end{proof}
\end{prop}
\begin{rmk}
In particular, if a Hermitian metric is ABC-geometrically formal,
then it is also geometrically Bott-Chern formal and the Bott-Chern harmonic
forms with respect to such metric have constant pointwise norms and inner
products. 
\end{rmk}
On a geometrically Aeppli formal manifold, the harmonic representatives of Dolbeault, Bott-Chern, and Aeppli cohomologies coincide and they induce a bigrading of the space of complex-valued de Rham harmonic forms. We obtain easily the following.
\begin{prop}\label{prop:aep-bot-const-norm}
On a geometrically Aeppli formal manifold $(M,J,g)$, the inner product of any two bigraded harmonic forms is constant and, in particular, every bigraded harmonic form has constant pointwise norm.
\end{prop}

As a consequence of Proposition \ref{prop:dolb-const-norm}, Proposition \ref{prop:abc-form-const}, and Proposition \ref{prop:aep-bot-const-norm}, on a geometrically Dolbeault formal (respectively, ABC-geometrically formal, respectively, geometrically Aeppli formal) manifold,  Dolbeault and conjugate-Dolbeault (respectively, Bott-Chern and Aeppli, respectively, bigraded) harmonic forms that are linearly independent at a point in $M$, are linearly independent on $M$. Hence, we can orthonormalize the spaces $\calh_{\dbar}^{\bullet,\bullet}(M)$ and $\calh_{\del}^{\bullet,\bullet}(M)$ (respectively $\mathcal{H}^{\bullet,\bullet}_{BC}(M)$ and $\mathcal{H}_{A}^{\bullet,\bullet}(M)$, respectively, $\calh^{\bullet,\bullet}(M)$) via a Gram-Schmidt procedure  with constant coefficients. The following statement is analogous to \cite[Lemma 5]{Kot01} and the proof is identical.
\begin{lemma}\label{lemma:2}
Let $(M,J,g)$ be a geometrically Dolbeault formal (respectively, ABC-geometrically  formal, respectively, geometrically Aeppli formal) manifold. Fix an orthogonal basis of either $\calh_{\dbar}^{p,q}(M)$ or $\mathcal{H}_{\del}^{p,q}(M)$ (respectively, either $\mathcal{H}_{BC}^{p,q}(M)$ or $\mathcal{H}_A^{p,q}(M)$, respectively, of $\calh^{p,q}(M)$). Then a form $\alpha\in\mathcal{A}^{p,q}(M)$ is harmonic in the respective sense if and only if $\alpha$ can be expressed as linear combination with constant coefficients of the respective orthogonal basis.
\end{lemma}
In the next theorems, we prove that the existence of geometrically Dolbeault formal, ABC-geometrically formal, and geometrically Aeppli formal metrics is obstructed by the complex cohomologies of the manifold. Surprisingly, obstructions for geometrically Dolbeault formal, ABC-geometrically formal, and geometrically Aeppli formal metrics arise also from the topology of the underlying manifold, as for geometrically formal Riemannian manifolds.
\begin{thm}\label{thm:geom-dolb-top}
Let $(M,J,g)$ be a geometrically Dolbeault formal manifold of complex dimension $n$. Then,
$$
  h_{\dbar}^{p,0}(M)\cdot h_{\dbar}^{0,q}(M)\leq h_{\dbar}^{p,q}(M)\leq h^{p,q}(\mathbb{T}^{n}_{\C}).
$$
Moreover,
$$
b_k(M)\leq b_k(\mathbb{T}^{2n}).
$$

\begin{proof}
Since $g$ is geometrically Dolbeault formal, the map
\begin{equation*}
    \mu\colon \mathcal{H}_{\dbar}^{p,0}(M)\otimes \mathcal{H}_{\dbar}^{0,q}(M)\rightarrow \mathcal{H}_{\dbar}^{p,q}(M)
\end{equation*}
induced by the wedge product of forms is well defined and injective. Indeed, if $\alpha\in\mathcal{H}_{\dbar}^{p,0}$ and $\beta\in\mathcal{H}_{\dbar}^{0,q}(M)$, by Proposition \ref{prop:dolb-const-norm}, if $\alpha\neq 0$ and $\beta\neq 0$, then they are everywhere non-zero. We observe that for a fixed $x\in M$, we might write
\[
\alpha|_x=\sum_{|I|=p} \alpha_I dz_I|_x, \qquad \beta|_x=\sum_{|J|=q}\beta_J d\overline{z}_J|_x,
\]
where $I=(i_1,\dots,i_p)$ and $J=(j_1,\dots, j_q)$ are multi-indices and $\alpha_I$ and $\beta_J$ are complex numbers.

Now, since $\alpha|_x\neq 0$, there exist $I_0=(i_1,\dots, i_p)$ and $J_0=(j_1,\dots, j_q)$ such that $\alpha|_x(\frac{\del}{\del z_{I_0}}|_x)=\alpha_{I_0}\neq 0$ and $\beta|_x(\frac{\del}{\del \overline{z}_{J_0}}|_x)=\beta_{J_0}\neq 0$. Then, it holds that
\[
(\alpha\wedge \beta)|_x\left(\frac{\del}{\del z_{I_0}}|_x,\frac{\del}{\del \overline{z}_{J_0}}|_x\right)=\alpha|_x\left(\frac{\del}{\del z_{I_0}}|_x\right) \beta|_x\left(\frac{\del}{\del \overline{z}_{J_0}}|_x\right)=\alpha_{I_0}\beta_{J_0}\neq 0.
\]
Hence, by Proposition \ref{prop:dolb-const-norm}, we have  $\alpha\wedge \beta\neq 0$. Since by Lemma \ref{lemma:2}, Dolbeault harmonic $(p,0)$ and $(0,q)$-forms can be made orthogonal using a Gram-Schmidt procedure with constant coefficients, the map $\mu$ is injective. As a result
\[
 h_{\dbar}^{p,0}(M)\cdot h_{\dbar}^{0,q}(M)\leq h_{\dbar}^{p,q}(M). 
\]
By Proposition \ref{prop:dolb-const-norm} and Lemma \ref{lemma:2}, the dimensions of $\mathcal{H}_{\dbar}^{p,q}(M)$ can not exceed $\text{rk}\bigwedge^{p,q}$, i.e.,
$$
h_{\dbar}^{p,q}(M)\leq \textstyle h^{p,q}(\mathbb{T}^n_{\C})=\text{rk}\bigwedge^{p,q}(M).
$$
Finally, note that by the classical Fr\"olicher inequality \cite{Fro55} and the first part of the theorem, it holds that
\[
b_k(M)\leq \sum_{p+q=k} h_{\dbar}^{p,q}(M,J)\leq \sum_{p+q=k} h^{p,q}(\mathbb{T}^{n}_{\C})=b_k(\mathbb{T}^{2n}).
\]
\end{proof}
\end{thm}
\begin{thm}\label{thm-abc-geom-top}
Let $(M,J,g)$ be a ABC-geometrically formal manifold of complex dimension $n$. Then,
\begin{align*}
     h_{A}^{p,q}(M)\leq h^{p,q}(\mathbb{T}^n_{\C}),\qquad 
    &h_{BC}^{p,0}(M)\cdot h_{BC}^{0,q}(M)\leq h_{BC}^{p,q}(M)\leq h^{p,q}(\mathbb{T}^n_{\C}).
\end{align*}
Moreover
\begin{align*}
        &b_k(M)\leq b_k(\mathbb{T}^{2n}).
\end{align*}
\begin{proof}
The upper bounds on $h_A^{p,q}(M)$ and $h_{BC}^{p,q}(M)$, as in Theorem \ref{thm:geom-dolb-top}, are a consequence of Lemma \ref{lemma:2}. The lower bound on $h_{BC}^{p,q}(M)$ follows from the injectivity of the wedge product of Bott-Chern harmonic forms $(p,0)$ and $(0,q)$-forms, as in Theorem \ref{thm:geom-dolb-top}, since ABC-geometric formality implies geometrically Bott-Chern formality. Moreover, by the inequality à la Fr\"olicher in \cite{AT15},
\begin{align*}
2b_k(M)&\leq\sum_{p+q=k}(h_{BC}^{p,q}(M)+h_{A}^{p,q}(M))\leq \sum_{p+q=k}(h^{p,q}(\mathbb{T}^n_{\C})+h^{n-p,n-q}(\mathbb{T}^n_{\C}))\\
&=2\sum_{p+q=k}h^{p,q}(\mathbb{T}^n_{\C})=2b_k(\mathbb{T}^{2n}).
\end{align*}
\end{proof}
\end{thm}
\begin{defi}
We will call \emph{geometrically Bott-Chern (CN) formal metric} a geometrically Bott-Chern formal metric with Bott-Chern harmonic forms of constant pointwise norm. 
\end{defi}
It is immediate to check that also the Aeppli harmonic forms with respect to a geometrically Bott-Chern (CN) formal metric have constant pointwise norm.

\begin{rmk}
On all the known examples of compact complex solvmanifold surfaces \cite{AT15}, nilmanifolds \cite{ST24}, and Calabi-Eckmann manifolds (see Section \ref{sec:CE}), every geometrically Bott-Chern formal metric is also geometrically Bott-Chern (CN) formal and ABC-geometrically formal. It would be of great interest to investigate further on the links between these notions.
\end{rmk}
We note that the proof of Theorem \ref{thm-abc-geom-top} relies only on the fact that a ABC-geometrically formal manifold is geometrically Bott-Chern (CN) formal.  This yields immediately the following.
\begin{cor}\label{cor:geom-BC-form-const-norm}
Let $(M,J,g)$ be a geometrically Bott-Chern (CN) formal manifold of complex dimension $n$. Then,
\[
h_A^{p,q}(M)\leq h^{p,q}(\mathbb{T}^n_{\C}), \quad h_{BC}^{p,0}(M)\cdot h_{BC}^{0,q}(M)\leq h_{BC}^{p,q}(M)\leq h^{p,q}(\mathbb{T}_{\C}^n).
\]
Moreover
\[
b_k(M)\leq b_k(\mathbb{T}^{2n}).
\]
\end{cor}

For geometrically Aeppli formal manifolds, we obtain the analogous bounds for any bigraded harmonic form.
\begin{thm}\label{thm:geom-aeppl-top}
Let $(M,J)$ be a compact complex manifold of complex dimension $n$ admitting a geometrically Aeppli formal metric $g$. Then, 
\begin{enumerate}
\item  $ h^{p,0}(M)\cdot h^{0,q}(M)\leq  h^{p,q}(M)\leq h^{p,q}(\mathbb{T}^n_{\C})$.\\[1pt]
\item $b_k(M)\leq b_k(\mathbb{T}^{2n})$.
\end{enumerate}
\end{thm}

From \cite[Theorem 6]{Kot01}, Theorem \ref{thm:geom-dolb-top}, and Corollary \ref{cor:geom-BC-form-const-norm}, we immediately have the following.
\begin{thm}\label{thm:blow-up-torus}
Let $M$ be a complex torus and let $\pi\colon \tilde{M}\rightarrow M$ be the blow-up of $M$
along a compact complex submanifold $Y$ of complex codimension $k\geq 2$. Then, $\tilde
{M}$ is not geometrically formal, nor geometrically Dolbeault formal, nor geometrically
Bott-Chern (CN) formal.
\begin{proof}
Let $k:=\operatorname{codim}_\C Y$ and let $E:=\pi^{-1}(Y)\cong \mathbb{P}(\mathcal{N}_{Y/M})$. From the formulas for the cohomology of a blow-up, see e.g., \cite[p.605]{GriHar},
\[
H^\bullet(\tilde{M};\C)\cong H^\bullet(M;\C)\oplus H^\bullet(E)/\pi^*H^{\bullet}(Y),
\]
it is easy to see that $b_{2k-2}(\tilde{M})>b_{2k-2}(M)$. We can conclude then by \cite[Theorem 6]{Kot01}, Theorem \ref{thm:geom-dolb-top}, and Corollary \ref{cor:geom-BC-form-const-norm}.
\end{proof}
\end{thm}
\begin{rmk}
Combining with Corollary \ref{cor:geom-BC-form-const-norm} and Lemma \ref{lemma:geom-aeppli-form-def}, the blow-up of any complex torus along any compact complex submanifold of
complex codimension at least $2$ is also not ABC-geometrically formal nor geometrically
Aeppli formal.
\end{rmk}

Kotschick \cite[Theorem 7]{Kot01} proved the following structure theorem for  geometrically formal manifolds.
\begin{thm}\label{thm:kot_subm}
Suppose the closed oriented manifold $M^n$ is geometrically formal. If $b_1(M)= k$,
then there is a smooth submersion $\pi\colon M \rightarrow  \mathbb{T}^k$, for which $\pi^{*}$ is an injection of cohomology algebras. In particular, if $b_1(M)= n$, then $M$ is diffeomorphic to $\mathbb{T}^n$. In
this case every formal Riemannian metric is flat.
\end{thm}
We now consider $M$ a geometrically Bott-Chern formal manifold of complex dimension $n$. We assume that the Bott-Chern harmonic $(1,0)$-forms have constant pointwise norm. We briefly recall the construction of the Albanese torus (see \cite{Bla56} for more details). Let $V$ be the complex vector space of holomorphic closed $(1,0)$-forms on $M$, i.e., $V=H^{1,0}_{BC}(M)$. The first homology space $H_1(M;\Z)$ is mapped to $V^*$ by the application $F$ defined as
\[
(F(u),\alpha):=\int_u \alpha, \qquad u\in H_1(M;\Z), \quad \alpha\in V.
\]
Let $\Delta:=F(H_1(M;\Z))$.
Then, the {\em Albanese torus} of $M$ is defined as
\[
\text{Alb}(M):=\frac{V^*}{\overline{\Delta}},
\]
where $\overline{\Delta}$ is the smallest closed subgroup of $V^*$ containing $\Delta$ whose identity component is a complex vector subspace of $V^*$, see \cite[p.~163]{Bla56}. Note that $\text{Alb}(M)$ is a compact complex torus.

Set $k:=\dim_\C \text{Alb}(M)$. In particular,
\[
k\leq \dim_\C H_{BC}^{1,0}(M)\leq\frac{1}{2} \dim H^1_{dR}(M;\R),
\]
see \cite[pp.~162--163]{Bla56}. Moreover, if $b_1(M)=2\dim_\C H^{1,0}_{BC}(M)$, then $k=\dim_\C H^{1,0}_{BC}(M)$. Indeed, by \cite[p.~162]{Bla56}, the $\R$-linear extension $F_\R\colon H_1(M;\R)\to V^*$ of $F$ is surjective. Since $b_1(M)=2\dim_\C H^{1,0}_{BC}(M)$, the map $F_\R$ is a surjective $\R$-linear map between real vector spaces of the same dimension, hence an isomorphism. Therefore $\Delta$, being the image under $F_\R$ of the lattice $\im\big(H_1(M;\Z)\to H_1(M;\R)\big)$, is a lattice in $V^*$. Hence $\Delta$ is a closed subgroup of $V^*$ with trivial identity component, so that $\overline{\Delta}=\Delta$ and $k=\dim_\C V^*=\dim_\C H^{1,0}_{BC}(M)$. In particular, this is the case if $M$ is K\"ahler, since then $b_1(M)=2\dim_\C H^{1,0}_{BC}(M)$ by the Hodge decomposition, see \cite[p.~163]{Bla56}.
\begin{rmk}
We stress that the condition $k=b_1(M)/2$ does not imply, in general, that $M$ is K\"ahler. E.g., the Iwasawa manifold has $b_1=4$ and $k=\dim_\C H^{1,0}_{BC}=2$ (see, e.g., \cite{AT15}), and it is not K\"ahler.
\end{rmk}


The Albanese mapping $alb\colon M\rightarrow \text{Alb}(M)$ is defined by choosing a base point $p_0\in M$ and the mapping
$$
p\longmapsto\left[\left(\omega\mapsto\int_{p_0}^p\omega\right)\right]\in \text{Alb}(M),
$$
for every $\omega\in \calh_{BC}^{1,0}(M)\cong H_{BC}^{1,0}(M)$. The map $alb$ is well defined, since two paths from $p_0$ to $p$ differ by a loop, whose periods lie in $\Delta\subseteq\overline{\Delta}$. 

\begin{lemma}\label{lemma:albanese}
The map $alb$ is a surjective holomorphic submersion.
\begin{proof}
The differential of $alb$ acts as
$$
d(alb)_p(v)=[ \omega\mapsto \omega_p( v)]\in T_0(V^*/\bar{\Delta}),
$$
for every $p\in M$, $v\in T_p^{1,0}(M)$, and $\omega\in \calh_{BC}^{1,0}(M)$, and $\omega_p(v)$ is the standard perfect pairing between $T_p^{\ast 1,0}(M)$ and  $T_p^{1,0}(M)$.

That is, $d(alb)_p$ is the composition of the map
$$
\tilde{d}_p\colon T_p^{1,0}(M)\to V^*,\qquad v\mapsto(\omega\mapsto\omega_p(v)),
$$
with the projection $V^*\to T_0(V^*/\overline{\Delta})$, which is surjective. The map $\tilde{d}_p$ is the transpose of the evaluation map $eval_p\colon V\to T_p^{\ast1,0}(M)$, defined by $eval_p(\omega)=\omega_p$. By assumption, Bott-Chern harmonic $(1,0)$-forms have constant norm, hence, by polarization, constant inner products, so that an orthogonal basis of $\mathcal{H}_{BC}^{1,0}(M)$ at a point in $M$ is an orthogonal basis of $\mathcal{H}_{BC}^{1,0}(M)$ at every point of $M$. Hence, the map $eval_p$ is injective for every $p\in M$. In particular, $\dim_\C H^{1,0}_{BC}(M)\leq n$. Therefore $\tilde{d}_p$ is surjective, and so is $d(alb)_p$, for every $p\in M$. Moreover, $alb$ is holomorphic, being locally given by the integrals of the holomorphic forms $\omega\in V$. Since $alb$ is a submersion, it is open, so that $alb(M)$ is open and compact in the connected manifold $\text{Alb}(M)$, i.e., $alb$ is surjective.
\end{proof}
\end{lemma}
Moreover, the pullback of the Albanese map induces the injective morphism
\[
alb^*\colon H^{1,0}_{BC}(\text{Alb}(M))\hookrightarrow H^{1,0}_{BC}(M),
\]
since the pullback on $d$-closed $(1,0)$-forms is injective.
By conjugation, also 
$$alb^*\colon \calh_{BC}^{0,1}(\text{Alb}(M))\rightarrow \calh^{0,1}_{BC}(M,J)$$ is an injective morphism.

As a result, since $g$ is geometrically Bott-Chern formal on $M$ and Bott-Chern harmonic $(1,0)$-forms have constant norms and inner products, the product of the pullback of linearly independent Bott-Chern harmonic $(1,0)$-forms and $(0,1)$-forms on $\text{Alb}(M)$ is non-zero and still harmonic. More precisely, if $\{\omega_1,\dots,\omega_k\}$ is the image under $alb^*$ of a basis of $\calh_{BC}^{1,0}(\text{Alb}(M))$, then, since $eval_x$ is injective for every $x\in M$ by the proof of the Lemma \ref{lemma:albanese}, the covectors $(\omega_i)_x\in T_x^{\ast1,0}(M)$ and $(\overline{\omega}_j)_x\in T_x^{\ast0,1}(M)$ are linearly independent, hence so are the monomials $(\omega_I\wedge\overline{\omega}_J)_x$; therefore every non-trivial linear combination of the forms $\omega_I\wedge\overline{\omega}_J$ with constant coefficients is a Bott-Chern harmonic form vanishing nowhere. Hence the map
$$
alb^*\colon \calh_{BC}^{\bullet,\bullet}(\text{Alb}(M))\rightarrow \calh_{BC}^{\bullet,\bullet}(M)
$$
is injective, since $\calh_{BC}^{\bullet,\bullet}(\text{Alb}(M))$ is generated by the spaces $\calh_{BC}^{1,0}(\text{Alb}(M))$ and $\calh_{BC}^{0,1}(\text{Alb}(M))$. In particular,
\begin{equation}\label{eq:diseq_alb}
h_{BC}^{p,q}(M)\geq h_{BC}^{p,q}(\mathbb{T}^{k}_{\C}).
\end{equation}
Since $g$ is geometrically Bott-Chern formal, by Proposition \ref{prop:1}, the spaces $\calh_{BC}^{\bullet,0}(M)$ and $\calh_{\dbar}^{\bullet,0}(M)$, respectively, $\calh_{BC}^{0,\bullet}(M)$ and $\calh_{\del}^{0,\bullet}(M)$, coincide and, by assumption, Dolbeault harmonic $(1,0)$-forms and conjugate-Dolbeault harmonic $(0,1)$-forms have constant pointwise norm. Then, the Albanese map induces injective maps
\[
alb^*\colon \calh^{\bullet,0}(\mathbb{T}^{k}_{\C})\rightarrow \calh_{\dbar}^{\bullet,0}(M)
\]
and
\[
alb^*\colon \calh^{0,\bullet}(\mathbb{T}^{k}_{\C})\rightarrow \calh_{\del}^{0,\bullet}(M).
\]

Let us assume now that the geometrically Bott-Chern formal metric $g$ is K\"ahler. By \cite[Proposition 2.1]{TT17}, $g$ is also geometrically formal, hence Theorem \ref{thm:kot_subm} applies. Moreover, the submersion can be taken to be the Albanese map, which is holomorphic. Indeed, since $M$ is compact K\"ahler, by the above discussion, $\Delta$ is a lattice in $V^*$ and $\text{Alb}(M)=V^*/\Delta$ is a complex torus of dimension $b_1(M)/2$.  Let $\{\omega_1,\dots,\omega_k\}$ be a basis of $V$ and let $z_1,\dots,z_k$ be the linear coordinates on $V^*$ dual to it, so that $alb^*dz_i=\omega_i$. By the proof of Lemma \ref{lemma:albanese}, the forms $\omega_I\wedge\overline{\omega}_J$ are Bott-Chern harmonic. However, since $g$ is K\"ahler, the forms $\omega_I\wedge\overline{\omega}_J=alb^*(dz_I\wedge d\overline{z}_J)$ are $d$-harmonic, and every non-trivial linear combination of them vanishes nowhere, so that it defines a non-zero class in $H^\bullet_{dR}(M;\C)$. Since $H^\bullet_{dR}(\text{Alb}(M);\C)$ is represented by the invariant forms $dz_I\wedge d\overline{z}_J$, the map $alb^*\colon H^\bullet_{dR}(\text{Alb}(M);\C)\to H^\bullet_{dR}(M;\C)$ is injective.

If $k=b_1(M)/2=n$, then $M$ is biholomorphic to a complex torus and the geometrically Bott-Chern formal metris is K\"ahler and flat. Indeed, by Lemma \ref{lemma:albanese}, the map $alb$ is a holomorphic submersion. However, since $\dim_{\C}M=\dim_{\C}\text{Alb}(M)$, for any point $p\in M$, the map $d(alb)_p$ is an isomorphism, hence $alb$ is a local biholomorphism of $M$ onto $\text{Alb(M)}$. Since $alb$ is proper and surjective, it is a finite unramified cover. In particular, we have that $\Gamma':=\pi_1(M)\leq \pi_1(Alb(M))=:\Gamma$ with $|\Gamma:\Gamma'|<\infty$. Since the universal cover of $\text{Alb}(M)$ is $\C^n$, we obtain that $X\cong \C^n/\Gamma'$ as complex manifolds. Now, we have that $\text{rk}(\Gamma'=\text{rk}(\Gamma)$, hence $X$ is complex $n$-dimensional torus. Furthermore, the geometrically Bott-Chern formal metric $g$ admits a global holomorphic coframe given by the Bott-Chern harmonic forms. Such closed forms have constant coefficients and they are closed. Hence, the metric $g$ is K\"ahler and has vanishing curvature.

To summarize, we have obtained the following theorem.
\begin{thm}\label{thm:fibr-ABC}
Let $M$ be a compact complex manifold admitting a geometrically Bott-Chern formal metric such that the Bott-Chern harmonic $(1,0)$-forms have constant pointwise norm. Let $k:=\dim_{\C} \text{Alb} (M)$. Then, there exists a surjective holomorphic submersion $\pi\colon M\rightarrow \mathbb{T}^{k}_{\C}$ so that $\pi^*$ is injective in $H_{BC}^{\bullet,\bullet}$, $H_{\dbar}^{\bullet,0}$, and $H_{\del}^{0,\bullet}$. Moreover, if 
\begin{itemize}
    \item the geometrically Bott-Chern formal metric is K\"ahler, then $k:=b_1(M)/2$ and $M$ admits a holomorphic submersion onto a complex torus that is injective in de Rham cohomology with complex coefficients;
    \item $k=b_1(M)/2=n$, then $M$ is biholomorphic to a complex torus and the metric $g$ is flat.
\end{itemize}
\end{thm}
\begin{rmk}
If $M$ admits either a geometrically Bott-Chern and Dolbeault formal metric or an  ABC-geometrically formal metric, then Theorem \ref{thm:fibr-ABC} applies. For the former, Dolbeault and Bott-Chern harmonic $(p,0)$-forms coincide and by Proposition \ref{prop:dolb-const-norm}, they have constant pointwise norm. For the latter, by Proposition \ref{prop:abc-form-const}, every Bott-Chern harmonic form has constant pointwise norm.
\end{rmk}
\begin{rmk}
Homogenous manifolds, or locally homogeneous manifolds, for which Bott-Chern cohomology can be computed via invariant forms and which admit invariant geometrically Bott-Chern formal metrics, satisfy the hypotheses of Theorem \ref{thm:fibr-ABC}.
\end{rmk}

\section{Compact complex surfaces and blow-ups}\label{sec:cpt-coplx-srf-bl-up}
It was proved \cite{AT15} that compact complex solvmanifold surfaces, i.e., Inoue surfaces of type $S_\mathcal{M}$ and $S_{\pm}$, primary and secondary Kodaira surfaces, and hyperelliptic surfaces, are geometrically Bott–Chern formal with Bott-Chern harmonic forms of constant pointwise norm. Moreover, they  are all geometrically Dolbeault formal, except for primary Kodaira surfaces. In the same paper, it is shown that the Calabi-Eckmann complex structure on the Hopf surface is both geometrically Dolbeault and Bott-Chern formal.

Hyperelliptic surfaces are, in particular,  geometrically formal K\"ahler \cite{ST24}, whereas none of the others satisfy the $\del\dbar$-lemma, hence they are not geometrically Aeppli formal. It can be checked directly that they are ABC-geometrically formal.

In this section, we start by applying the obstructions of the previous section to investigate the existence of  geometrically formal Hermitian metrics on classes of minimal compact complex surfaces, beyond the locally homogeneous setting.
\begin{thm}\label{thm:rational}
Hermitian metrics on Hirzebruch surfaces $\Sigma_n$, with $n\geq2$ odd, are never geometrically Aeppli formal.
\begin{proof}
A minimal rational surface $X$ is either $\mathbb{CP}^2$ or Hirzebruch surfaces $\Sigma_n$, with $n=0$ or $n\geq 2$. We observe that $\Sigma_1$ is isomorphic to the blow-up of $\mathbb{CP}^2$ at a point and, hence, it is not minimal. Since $b_1(X)=0$, they are always K\"ahler. Moreover, $b_2=1$ for $\mathbb{CP}^2$ and $b_2=2$ for $\Sigma_n$.

Clearly, $\mathbb{CP}^2$ and $\Sigma_0=\mathbb{CP}^1\times \mathbb{CP}^1$ are geometrically formal K\"ahler, since $\mathbb{CP}^2$ and $\mathbb{CP}^1$ are Hermitian symmetric spaces.

If $n $ is odd, it is well known that the diffeomorphism type of $\Sigma_n$ is that of the non-trivial $S^2$-bundle over $S^2$. Then $\Sigma_n$ has no geometrically Aeppli formal metrics. Indeed, assume by contradiction, that there exists a geometrically Aeppli formal metric on X. Then, it induces a geometrically formal metric on $\Sigma_n$, which is a contradiction  by \cite[Lemma 9]{Kot1}.
\end{proof}   
\end{thm}
\begin{rmk}
Since Hirzebruch surfaces are K\"ahler and, hence, satisfy the $\del\dbar$-lemma, by Lemma \ref{lemma:geom-aeppli-form-def}, Hirzebruch surfaces $\Sigma_n$, with $n$ odd, admit no ABC-geometrically formal metrics.
\end{rmk}
Ruled surfaces of genus $0$ are Hirzebruch surfaces, hence Theorem \ref{thm:rational} applies. Moreover, by \cite[Theorem 4]{Kot1} ruled surfaces of genus 1 admit geometrically formal metrics. If such metrics are compatible with the complex structures and K\"ahler, then such a ruled surface is geometrically formal K\"ahler. 
\begin{thm}\label{thm:ruled}
Ruled surfaces of genus $g\geq 2$, class VII minimal surfaces with $b_2>6$, classical Enriques surfaces, and K3 surfaces are not geometrically formal, nor geometrically Dolbeault formal, nor geometrically Bott-Chern (CN) formal.
\begin{proof}
Let $X$ be a ruled surface. Then, it is well known that $X$ is K\"ahler, 
$$h^{1,0}(X)=h^{0,1}(X)=g,$$ where $g$ is the genus of $X$, and 
$$h^{1,1}(X)=b_2(X)=2.$$

We first consider $g=2$, so that $b_1(X)=2g=4$. If $X$ admits a geometrically formal metric, then $X$ is diffeomorphic to a torus by Theorem \ref{thm:kot_subm},  which yields a contradiction. Moreover, if $X$ admits a geometrically Dolbeault formal metric, since
$$
h^{1,0}(X)\cdot h^{0,1}(X)=4>2=h^{1,1}(X),
$$
we obtain a contradiction by Theorem \ref{thm:geom-dolb-top}. Finally, if $X$ admits a geometrically Bott-Chern formal metric with Bott-Chern harmonic $(1,0)$-forms with constant pointwise norm, since $X$ is K\"ahler, $\text{Alb}(X)$ is a complex torus of dimension $b_1(X)/2=2$, see, e.g., \cite{GriHar}, i.e., $k=b_1(X)/2=2=\dim_\C X$; hence, by Theorem \ref{thm:fibr-ABC}, $X$ is biholomorphic to a complex torus, which forces a contradiction, since $b_2(X)=2\neq 6=b_2(\mathbb{T}^4)$. Hence, ruled surfaces of genus $g=2$ admit no geometrically formal, nor geometrically Dolbeault formal, nor geometrically Bott-Chern formal (CN) metrics.

Assume now that the genus $g>2$. But then, 
$$b_1(X)=2g>4=b_1(\mathbb{T}^4).$$ 
Hence, by \cite[Theorem 6]{Kot01}, Theorem \ref{thm:geom-dolb-top}, and Corollary \ref{cor:geom-BC-form-const-norm}, the thesis follows.

If $X$ is either any class VII surface with $b_2> 6$, or a $K3$ surface, or a classical Enriques surface $E$, then
\[
b_2(K_3)=22, \qquad b_2(E)=10.
\]
Since $b_2(\mathbb{T}^4)=6$, the thesis follows again by \cite[Theorem 6]{Kot01}, Theorem  \ref{thm:geom-dolb-top} and Corollary \ref{cor:geom-BC-form-const-norm}.
\end{proof}
\end{thm}
\begin{rmk}
By \cite[Corollary 3]{PSZ24}, ruled surfaces of genus $g\geq 2$ always admit non-vanishing triple ABC-Massey products, hence they are also not geometrically Bott-Chern formal.
\end{rmk}
\begin{rmk}
The case of geometrically formal metrics on $K3$ was discussed in detail by Kotschick \cite{Kot01} and Huybrechts \cite{Huy00}.
\end{rmk}
Concerning the existence of geometrically formal Hermitian  metrics on compact K\"ahler non-minimal surfaces, or, more in general, on blow-ups of compact K\"ahler manifolds, we prove the following. 
\begin{thm}\label{thm:blow-up-kahler}
On the blow-up of any compact K\"ahler manifold along any compact complex submanifold of complex codimension strictly greater than one, there exist no geometrically formal K\"ahler metrics.
\begin{proof}
Let $(M,g,\omega)$ be an $n$-dimensional compact K\"ahler manifold and let $Y$ be a compact complex submanifold with  $\text{codim} Y_\C>1$.

Denote by $$\mathcal{K}_M:=\{a\in H^{1,1}_\R(M):\exists \alpha\in[a]\,\,\ \text{s.t.}\, \alpha\,\, \text{is a K\"ahler metric on}\,\, M\}$$ the {\em K\"ahler cone} and by $\mathcal{C}_M$ the {\em positive cone}, i.e., the connected component of $\{[\alpha]\in H^{1,1}_{\R}:\int_M\alpha^n>0\}$ containing $\mathcal{K}_M$.

Let $\pi\colon \tilde{M}\rightarrow M$ be the blow-up of $M$ along $Y$. We claim that
\[
\mathcal{K}_{\tilde{M}}\subsetneq \mathcal{C}_{\tilde{M}}.
\]
Indeed, let us consider $\pi^*\omega$. Then,
\[
\int_{\tilde{M}}(\pi^*\omega)^n>0,
\]
i.e., $\pi^*\omega\in\mathcal{C}_{\tilde{M}}$. On the other hand, setting $E:=\pi^{-1}(Y)$, we see that $\pi^*\omega(v,\tilde{J}v)=0$, for every $v\in \ker(d\pi)_p\subset T_pE$ and $p\in E$, hence $\pi^*\omega\notin \mathcal{K}_{\tilde{M}}$.

From a result of Huybrechts \cite[Proposition 3.2]{Huy00}, we conclude that there cannot exist geometrically formal K\"ahler metrics on $\tilde{M}$.

\end{proof}
\end{thm}

\begin{rmk}
Huybrechts \cite[§6]{Huy00} proved that no K\"ahler metric on the blow-up of a projective variety can be geometrically formal. Theorem \ref{thm:blow-up-kahler} above aligns with this result and focuses on a broader class of smooth manifolds.
\end{rmk}

Strong cohomological properties, such as the $\del\dbar$-lemma, do not necessarily imply the existence of  geometrically Hermitian formal metrics.

Indeed, we provide an example of a compact non-K\"ahler complex manifold which satisfies the $\del\dbar$-lemma, hence is Dolbeault and Sullivan formal, but is not geometrically formal, nor geometrically Dolbeault formal, nor geometrically Bott-Chern formal.
\begin{exa}\label{example:1}
We recall the construction of \cite[§ 7]{SfeTom22}. Let $M$ be the Iwasawa manifold, i.e., $M$ is the complex parallelisable manifold $M=\Gamma\backslash G$, where
\[
G=\left\{\begin{pmatrix}
    1 & z_1  & z_3\\
    0 & 1 & z_2\\
    0 & 0 &1
\end{pmatrix}:z_i\in\C, i\in\{1,2,3\}\right\},\qquad\Gamma=G\cap GL(3,\Z[i]).
\]
Equivalently, the complex Lie group $G$ and its subgroup $\Gamma$ can be identified with $(\C^3,\star)$ and $(\Z[i]^3,\star)$ respectively, where
\[
(z_1,z_2,z_3)\star(w_1,w_2,w_3):=(z_1+w_1,z_2+w_2,z_3+z_1w_2+w_3)
\]
The Lie group automorphism of $(\C^3,\star)$ defined by $$\sigma(z_1,z_2,z_3):=(iz_1,iz_2,-z_3)$$  descends to a holomorphic automorphism of $M$ and acts with $16$ fixed points. The global-quotient type orbifold $\hat{M}:=M/\langle\sigma\rangle$ satisfies the $\del\dbar$-lemma \cite[Lemma 7.4]{SfeTom22}, and has non-vanishing Bott-Chern numbers
\begin{align*}
    h_{BC}^{0,0}(\hat{M})=1, \quad h_{BC}^{1,1}(\hat{M})=4, \quad h_{BC}^{3,0}(\hat{M})=1,\\
    h_{BC}^{0,3}(\hat{M})=1, \quad h_{BC}^{2,2}(\hat{M})=4, \quad h_{BC}^{3,3}(\hat{M})=1.
\end{align*}
The resolution $\tilde{M}$ of $\hat{M}$ is explicitly constructed by, first, blowing up the fixed loci of  $\psi:=\sigma^2$ on $M$, consisting of $8$ disjoint complex $1$-dimensional tori $\mathcal{C}_i$ on $M$. Then, we take the quotient with respect to the induced action $\tilde{\psi}$.  Each exceptional divisor $E_i$ of the blow-up map $\pi$ is isomorphic to $\mathbb{P}(\mathcal{N}_{\mathcal{C}_i|M})$, which is a trivial bundle $\mathcal{C}_i\times \mathbb{CP}^1$.
On the resulting manifold, we perform a further blow-up $\epsilon$ along the fixed locus of the induced action of $\sigma$, which is the disjoint union of sixteen copies of $\mathbb{CP}^1$, and then we take the quotient with respect to the induced action of $\tilde{\sigma}$. Each exceptional divisor $E'$ is a $\mathbb{CP}^1$-bundle over $\mathbb{CP}^1$ and has the same Bott--Chern numbers as $\mathbb{CP}^1\times\mathbb{CP}^1$.

This procedure yields a smooth compact complex non-K\"ahler (see [64, Remark~4.4.6]) manifold $\tilde M$ which satisfies the $\partial\bar\partial$-lemma [66, Theorem~7.1] and, hence, is formal according to Sullivan and strictly Dolbeault formal in the sense of Neisendorfer and Taylor [55].

In particular, $\tilde M$ is obtained by blowing up first along eight disjoint elliptic curves and then along sixteen disjoint rational curves, taking the quotient by the induced involution after each blow-up.

From the formulas for Bott--Chern cohomology of the blow-up of a threefold, see, e.g., [69, Theorem~5.1] and [77, Equation~(4.3)], we have that
\[
H_{BC}^{p,q}(\tilde{M})=\epsilon^*(\pi^*H_{BC}^{p,q}(M)^{\tilde{\psi}}\oplus H_{BC}^{p,q}(E)/\pi^*(H_{BC}^{p,q}(Fix_{\psi})))^{\tilde{\sigma}}\oplus H_{BC}^{p,q}(E')/\epsilon^*H_{BC}^{p,q}(Fix_{\sigma}).
\]
We can then compute that
\[
\begin{aligned}
h^{0,0}_{BC}(\tilde M)&=1,
& h^{1,1}_{BC}(\tilde M)&=26,
& h^{3,0}_{BC}(\tilde M)&=1,\\
h^{0,3}_{BC}(\tilde M)&=1,
& h^{2,2}_{BC}(\tilde M)&=26,
& h^{3,3}_{BC}(\tilde M)&=1,\\
h^{2,1}_{BC}(\tilde M)&=2,
& h^{1,2}_{BC}(\tilde M)&=2,
\end{aligned}
\]
from which
\[
\begin{gathered}
b_0(\tilde M)=1,\qquad
b_1(\tilde M)=0,\qquad
b_2(\tilde M)=26,\qquad
b_3(\tilde M)=6,\\
b_4(\tilde M)=26,\qquad
b_5(\tilde M)=0,\qquad
b_6(\tilde M)=1.
\end{gathered}
\]
Since
\[
9=h^{1,1}(\mathbb T^6)
<
h^{1,1}_{BC}(\tilde M)
=
h^{1,1}_{\bar\partial}(\tilde M)
=
h^{1,1}_{A}(\tilde M),
\qquad
15=b_2(\mathbb T^6)<b_2(\tilde M),
\]
from \cite[Theorem 6]{Kot01} and Theorem \ref{thm:geom-dolb-top}, the manifold $\tilde{M}$ does not admit geometrically formal, nor geometrically Dolbeault formal metrics. The non-existence of geometrically Bott-Chern formal metric is a direct consequence of the existence of a non-vanishing triple ABC-Massey product \cite[Theorem 7.1]{SfeTom22}.
\end{exa}
In the next example, we start from a completely solvable Nakamura manifold which is geometrically Aeppli formal (hence, a $\del\dbar$-manifold), and by the same process of orbifold resolution, we produce a new compact complex manifold satisfying the $\del\dbar$-lemma but admitting no geometrically formal, geometrically Bott-Chern (CN) formal, nor geometrically Dolbeault formal metric. 
\begin{exa}\label{example:2}
As in \cite[Construction I]{KasSte}, let $G=\C\ltimes_{\phi}\C^2$ with $\phi(x+\sqrt{-1}y)=\text{diag}(e^{x},e^{-x})\in GL(2;\C)$. Then, pick
\[
A:=\begin{pmatrix}
    2 & 1\\
    1 & 1
\end{pmatrix}\in SL(2;\Z)
\]
so that, w $ P=\begin{pmatrix} 1 & \frac{-1+\sqrt{5}}{2}\\  \frac{-1+\sqrt{5}}{2} & -1\end{pmatrix}$, we have
\[
P^{-1}\cdot A\cdot P=\text{diag}\left(\frac{3+\sqrt{5}}{2},\frac{3-\sqrt{5}}{2}\right).
\]
Let now $a\in \R$ be such that $e^a=\frac{3+\sqrt{5}}{2}$.
Set $\Delta:=P^{-1}\Z^2\oplus iP^{-1}\Z^2$ and $\Lambda:=3a\Z\oplus b\sqrt{-1}\Z$, $0\neq b\in \R$. Then, $\Gamma:=\Lambda\ltimes_{\phi} \Delta$ is a cocompact lattice for $G$, and the manifold $X:=\Gamma\backslash G$ is a said a completely solvable Nakamura manifold. We choose $b\notin \pi\Z$. The set of left-invariant $1$-forms on $G$ given by
\[
\eta^1=dz_1, \qquad \eta^2=e^{-\frac{z_1+\overline{z}_1}{2}}dz_2, \qquad \eta^3=e^{\frac{z_1+\overline{z}_1}{2}}dz_3
\]
define a invariant complex structure on $X$ by declaring that $\bigwedge^{1,0}\mathfrak{g}_{\C}^*=\C\langle\eta^1,\eta^2,\eta^3\rangle$. From \cite{KasSte}, $X$ satisfies the $\del\dbar$-lemma and the invariant metric $\omega=\frac{i}{2}\sum_{i=1}^3\eta^i\wedge\overline{\eta}^i$ is geometrically Aeppli formal, see Theorem  \ref{thm:geom-aeppli-metaabelian}. In particular, the space of bigraded harmonic representatives with respect to $\omega$ is generated, as an algebra, by the forms $\langle1,\phi^1,\overline{\phi}^1,\phi^{23}, \phi^{2\overline{3}}, \phi^{3\overline{2}},\phi^{\overline{23}}\rangle$.
The involution $\psi\colon \C^3\rightarrow \C^3$ defined by $\psi(z_1,z_2,z_3):=(z_1,-z_2,-z_3)$ descends to an involution on $X$, which fixes $16$ complex disjoint tori. Let us then denote $\hat{X}:=X/\langle\psi\rangle$ the orbifold thus obtained. A smooth resolution \cite{KasSte} is provided by $Bl_{Fix_{\psi}}X/\langle\tilde{\psi}\rangle$, where $\tilde{\psi}$ is the action induced by $\psi$ on the blow-up. Combining \cite{ASTT} and \cite{YanYan20}, we have that
\[
H_{BC}^{\bullet,\bullet}(Bl_{Fix_{\psi}}X/\langle\tilde{\psi}\rangle)=\pi^{\ast}H_{BC}^{\bullet,\bullet}(M)^{\tilde{\psi}}\oplus \left(H_{BC}^{\bullet,\bullet}(E)/\pi^*H_{BC}^{\bullet,\bullet}(Fix_{\psi})\right)^{\tilde{\psi}},
\]
where we set $E=\pi^{-1}(Fix_{\psi})$.
However, the action of $\tilde{\psi}$ commutes with the projection $\pi$, hence $\pi^*\alpha$ is $\tilde{\psi}$-invariant only if $\alpha$ is $\psi$-invariant. Since all the generators of the space of harmonic forms on $X$ are $\psi$-invariant, we have that $\pi^*H_{BC}^{\bullet,\bullet}(M)^{\tilde{\psi}}=\pi^*H_{BC}^{\bullet,\bullet}(M)$. Furthermore, $\tilde{\psi}$ and $\psi$ act as the identity on $E$ and $Fix_{\psi}$ respectively, hence $H_{BC}^{\bullet,\bullet}(E)/\pi*H_{BC}^{\bullet,\bullet}(Fix_{\psi})$ is also $\tilde{\psi}$-invariant. This implies that
\[
H_{BC}^{\bullet,\bullet}(Bl_{Fix_{\psi}}X/\langle\tilde{\psi}\rangle)\cong H_{BC}^{\bullet,\bullet}(Bl_{Fix_{\psi}}X)\cong H_{BC}^{\bullet,\bullet}(X)\oplus H_{BC}^{\bullet,\bullet}(E)/\pi^*H_{BC}^{\bullet,\bullet}(Fix_{\psi}).
\]
As a result, since $X$ satisfies the $\del\dbar$-lemma and this property is invariant under the blow-up operation if the center of the blow-up satisfies the $\del\dbar$-lemma \cite{Ste21}, also $Bl_{Fix_{\psi}}/\langle\tilde{\psi}\rangle$ satisfies the $\del\dbar$-lemma. Let us denote $Bl_{Fix_{\psi}}X/\langle\tilde{\psi}\rangle$ by $X'$. By \cite[Formula (4.3)]{YanYan20}, since we are blowing up $16$ disjoint complex $1$-dimensional tori in $X$, we obtain immediately that 
$$h_{BC}^{p,0}(X')=h_{BC}^{p,0}(X),\quad h_{BC}^{0,q}(X')=h_{BC}^{0,q}(X),$$ 
$$h_{BC}^{p,3}(X')=h_{BC}^{p,3}(X),\quad h_{BC}^{3,q}(X')=h_{BC}^{3,q}(X),$$ and
\[
h_{BC}^{1,1}(X')=h_{BC}^{2,1}(X')=h_{BC}^{1,2}(X')=h_{BC}^{2,2}(X')=19.
\]
Since, e.g., $$h^{1,1}(\mathbb{T}_\C^3)=9<19=h^{1,1}(X'),$$ the manifold $X'$ cannot admit any geometrically Bott-Chern (CN) formal metrics by Theorem \ref{cor:geom-BC-form-const-norm}, and hence no geometrically Aeppli formal metrics. Since $X'$ satisfies the $\del\dbar$-lemma, the same bounds imply that $X'$ admits no geometrically formal metrics nor geometrically Dolbeault formal metrics, by \cite[Theorem 6]{Kot01} and Theorem \ref{thm:geom-dolb-top}.
\end{exa}

\begin{rmk}
Theorem \ref{thm:blow-up-kahler} proves that the blow-up operation disrupts the geometric formality of any K\"ahler metric of the base. On the other hand, the examples above, e.g., $K3$ surfaces, Example 1 and Example 2, and any blow-up of a complex torus (Theorem \ref{thm:blow-up-torus}), highlight that the existence of geometrically formal and geometrically formal Hermitian metrics on the blow-up of a compact complex manifold is obstructed also by the topology of the manifold.
\end{rmk}
\begin{rmk}
Example 1 and the example constructed in \cite[§4]{PSZ24}, which is obtained as a sequence of blow-ups of a geometrically Bott-Chern formal manifold, i.e., $\mathbb{CP}^3$, are both manifolds which admit non-vanishing triple ABC-Massey products, and hence, no geometrically Bott-Chern formal metrics. It would then be interesting to find simple topological obstructions also for the geometric Bott-Chern formality of blow-ups of compact complex manifolds.
\end{rmk}

\section{Curvature techniques for  geometrically formal K\"ahler manifolds}\label{sec:weitz}
In this short section, we provide existence results for geometrically formal K\"ahler metrics exploiting nonnegative curvature properties. 

The argument relies on the Gallot-Meyer theorem \cite{GM75}, which proves that on a smooth compact manifold with nonnegative  curvature operator, every harmonic form is parallel with respect to the Levi-Civita connection. Indeed, the de Rham Laplacian acting on a smooth $k$-form on a $2n$-dimensional smooth compact manifold $M$ can be written as the Lichnerowicz Laplacian 
\[
\Delta_{dR}\alpha=\nabla^*\nabla\alpha+\text{Ric}(\alpha),
\]
where $\text{Ric}(\alpha)$ is the \emph{Weitzenb\"ock curvature operator} given by
\[
\text{Ric}(\alpha)(X_1,\dots, X_k):=\sum_{i=1}^k\sum_{j=1}^{2n}(R(e_j,X_i)\alpha)(X_1,\dots, e_j, \dots, X_k).
\]
If $\alpha$ is harmonic, i.e., $\Delta_{dR}\alpha=0$ or, equivalently,  $\nabla^*\nabla\alpha=-\text{Ric}(\alpha)$, we have the Bochner formula 
\[
\frac{1}{2}\Delta|\alpha|^2=||\nabla\alpha||^2-g(\nabla^*\nabla\alpha,\alpha)=||\nabla\alpha||^2+g(\text{Ric}(\alpha),\alpha)\geq 0,
\]
where the last inequality stems from the assumption that $\mathcal{R}\geq 0$. By the maximum principle, we obtain $\Delta|\alpha|^2=\nabla\alpha=0$, i.e., $\alpha$ is parallel.

\begin{thm}\label{thm:curv-kahl-non-neg-curv}
Let $M$ be a compact complex manifold with a K\"ahler metric $g$ with nonnegative curvature operator. Then $g$ is geometrically formal and $M$ is strongly formal.
\end{thm}
\begin{proof}
Let $M$ be a compact complex manifold endowed with a K\"ahler metric of non-negative curvature operator.

We recall that $\nabla^{LC}\alpha=0$ implies that $d\alpha=d^*\alpha=0$. Since $g$ is a K\"ahler metric on $M$, by the K\"ahler identities we have that
\[
\mathcal{H}_{dR}^{k}(M;\C)=\bigoplus_{p+q=k}\calh^{p,q}(M),
\]
hence we can consider de Rham harmonic forms with pure bidegree. Let then $\alpha\in \mathcal{H}^{p,q}(M),\beta\in \calh^{r,s}(M)$. Then, since $g$ has nonnegative curvature operator, $\nabla^{LC}\alpha=\nabla^{LC}\beta=0$ by \cite{GM75}. By the Leibniz rule then
\[
\nabla^{LC}(\alpha\wedge \beta)=\nabla^{LC}\alpha\wedge \beta + \alpha\wedge\nabla^{LC}\beta=0,
\]
which implies that 
\[
d(\alpha\wedge \beta)=0, \qquad d^*(\alpha\wedge \beta)=0,
\]
meaning that $\alpha\wedge \beta\in \calh^{p+r,q+s}(M)$, i.e., the K\"ahler metric $g$ is geometrically formal. As a result, by Proposition \ref{prop:0}, $M$ is also strongly formal.
\end{proof}
We immediately obtain the following result.
\begin{cor}\label{cor:kahler-flat}
Compact K\"ahler flat manifolds are geometrically formal and strongly formal.
\end{cor}
Compact K\"ahler flat manifolds are characterized in \cite{Rog} as quotients of complex tori by finite groups acting holomorphically and freely.
\begin{cor}\label{cor:kahler-solv}
Let $(M,J)$ be a K\"ahler solvmanifold. Then, $(M,J)$ is geometrically formal and strongly formal.
\begin{proof}
By \cite{Has06}, every K\"ahler solvmanifold has the structure of a quotient of a complex torus by a finite group of holomorphic isometries. In particular, $(M,J)$ inherits the invariant Hermitian flat metric $g$ of the complex torus covering it. Theorem \ref{thm:curv-kahl-non-neg-curv} then immediately applies.
\end{proof}
\end{cor}
\begin{rmk}
For the proof of the Gallot-Meyer theorem, it actually suffices to assume that the Weitzenb\"ock curvature operators are nonnegative. It would be interesting to investigate whether this weaker curvature assumption allows for a generalization of Theorem \ref{thm:curv-kahl-non-neg-curv} to a larger class of manifolds. 
\end{rmk}

\section{Dolbeault  and Bott-Chern formality of complex solvable manifolds}

Let $M=\Gamma\backslash G$ be a complex parallelisable manifold, i.e., the compact quotient of a simply connected complex Lie group $G$ by a discrete subgroup $\Gamma$ \cite{Wan54}. Following the notation of Nakamura \cite{Nak75}, if the complex Lie group $G$ is solvable (respectively, nilpotent), we call $M$ a \emph{complex
solvable (respectively, nilpotent) manifold}. 

In \cite{CFGU}, it is proved that some examples of complex nilpotent manifolds are not Dolbeault formal. We give a proof here for every complex nilpotent manifold.

\begin{thm}\label{thm:dol:cnm}
Complex nilpotent manifolds are not Dolbeault formal, unless they are complex tori.
\begin{proof}
From \cite[§2.7]{AK17}, the Dolbeault minimal model of a complex nilpotent manifold $M=\Gamma\backslash G$ is given by
\begin{gather}\label{eq:dolb_nil_par}
(\textstyle\bigwedge^{\bullet}(\mathfrak{g}_+^*),\dbar\equiv 0)\otimes_\C (\bigwedge^{\bullet}(\mathfrak{g}_-^*),\dbar)\hookrightarrow (\mathcal{A}^{\bullet,\bullet}(M),\dbar)\\
\textstyle\bigwedge^{\bullet}(\mathfrak{g}_+^*)\otimes_\C H_{\dbar}^{\bullet}(\bigwedge^\bullet(\mathfrak{g}_-^*))\cong H_{\dbar}^{\bullet,\bullet}(M),\nonumber
\end{gather}
where $\mathfrak{g}=\mathfrak{g}_+\oplus \mathfrak{g}_-$ is the decomposition of the Lie algebra of $G$ in terms of the spaces of holomorphic and anti-holomorphic left-invariant vector fields on $G$.

By definition of complex parallelisable manifolds, it holds that $\mathfrak{g}_-$ is a complex nilpotent Lie algebra. Hence,  unless $\mathfrak{g}$ is abelian, the algebra $(\bigwedge^{\bullet}(\mathfrak{g}_-^*),\dbar)$ is not Sullivan formal \cite{BG88,Has89} and, via the isomorphism
$$
\textstyle(\bigwedge^{\bullet} (\mathfrak{g}_-^*),\dbar)\cong (\bigwedge^{0,\bullet}(\mathfrak{g}_\C^*),\dbar),
$$
the algebra $(\bigwedge^{\bullet}(\mathfrak{g_-^*}),\dbar)$ is also not Dolbeault formal.

From \eqref{eq:dolb_nil_par}, we obtain that the Dolbeault model of $M$ is the tensor product of a Dolbeault formal algebra, i.e., $(\bigwedge^\bullet(\mathfrak{g}_+),\dbar\equiv 0)\cong(\bigwedge^{\bullet,0}(\mathfrak{g}_\C^*),\dbar\equiv 0)$, and a non-formal one, i.e., $(\bigwedge^{\bullet}(\mathfrak{g_-^*}),\dbar)$, unless $M$ is a complex torus. As a result, the manifold  $M$ is not Dolbeault formal, unless it is a complex torus.
\end{proof}
\end{thm}
Since Dolbeault formality obstructs the existence of geometrically Dolbeault metrics \cite[Proposition 2.1]{TomTor14}, we immediately obtain the following.
\begin{cor}\label{cor:no-geom-dolb-cn}
Complex nilpotent manifolds are not geometrically Dolbeault formal, unless they are complex tori.
\end{cor}
\begin{rmk}\label{rmk:geom-dolb-cs}
If $M$ is a complex solvable manifold, it may admit a geometrically Dolbeault formal metric, even if it is not a complex torus. Indeed, the complex parallelisable Nakamura manifold admits a lattice \cite[§5.2, $b,c\notin \pi\Z$]{Kas} such that the diagonal invariant Hermitian metric is geometrically Dolbeault formal.
\end{rmk}

We focus now on geometrically Bott-Chern formality and weak formality of complex parallelisable manifolds.
\begin{thm}\label{thm:geom-bott-cpm}
Complex parallelisable manifolds are not geometrically Bott-Chern formal, unless they are complex tori.
\begin{proof}
Let $M=\Gamma\backslash G$ be a complex parallelisable manifold. By definition, $M$ admits a global coframe of $G$-invariant holomorphic $(1,0)$-forms, see \cite{Wan54}. If we suppose that there exists a geometrically Bott-Chern formal metric on $M$, then every holomorphic $(1,0)$-form of this coframe is closed by Proposition \ref{prop:1}, i.e., $G$ is abelian and $M$ is a complex torus.  
\end{proof}
\end{thm}

In \cite{Nak75}, Nakamura gave a characterization of complex solvable manifolds up to complex dimension $5$. For the complex Lie groups numbered as in the list \cite[pag. 108--109]{Nak75}, for which the existence of a lattice is clear, we prove the following.
\begin{prop}\label{prop:ABC-Massey-cs}
Let $M$ be a complex solvable manifold of either type $III.2,3$,  $IV.2-4,6$ or $V.2-10,12,15,17$. Then $M$ admits a non-vanishing triple ABC-Massey product.
\begin{proof}
We construct a non-vanishing ABC-Massey product for each case. Throughout the proof, we will always fix the invariant Hermitian metric on $M$ for which the fixed coframe of $(1,0)$-forms is orthonormal, namely, the Hermitian metric whose fundamental form is given by
\[
\omega=\frac{i}{2}\sum_{j=1}^n \phi^j\wedge \phi^{\overline{j}}.
\]
We briefly recap the strategy. We first construct a non-vanishing triple ABC-Massey product at the invariant level, i.e., on the Lie algebra $(\mathfrak{g},J)$ of $G$. Then, thanks to \cite[Lemma 2.1]{ST24}, we obtain a corresponding non-vanishing triple ABC-Massey product on $(M,J)$.

In particular, for each case, we start by studying the space $$\im(\del\dbar|_{\bigwedge^{\bullet,\bullet}(\mathfrak{g})})=\im(\del|_{\bigwedge^{\bullet,0}})\otimes\im(\dbar|_{\bigwedge^{0,\bullet}(\mathfrak{g})}),$$ where equality is ensured by $d(\bigwedge^{1,0}(\mathfrak{g}))\subset \bigwedge^{2,0}(\mathfrak{g})$. It turns out that
$$
\im(\del|_{\bigwedge^{\bullet,0}(\mathfrak{g})})\subset\calh_{BC}^{\bullet+1,0}(\mathfrak{g}), \qquad \im(\dbar|_{\bigwedge^{0,\bullet}(\mathfrak{g})})\subset\calh_{BC}^{0,\bullet+1}(\mathfrak{g})
$$
since $$\im(\del\dbar\ast|_{\bigwedge^{\bullet,0}(\mathfrak{g})})=\im(\del\dbar\ast|_{\bigwedge^{0,\bullet}(\mathfrak{g})})=\{0\}.$$

Thus, a natural choice of Bott-Chern cohomology classes for the construction of ABC-Massey products is  $[\del\sigma]\in H_{BC}^{|\sigma|+1,0}(\mathfrak{g})$ for a pure degree $\sigma\in\bigwedge^{\bullet,0}(\mathfrak{g})$, or its conjugate $[\dbar\overline{\sigma}]\in H_{BC}^{0,|\sigma|+1}$. In either case, $\del\sigma$ and $\dbar\overline{\sigma}$ are Bott-Chern harmonic.  We look for a choice of three Bott-Chern harmonic forms $\del\alpha\in \calh_{BC}^{p,0}(\mathfrak{g})$, $\dbar\overline{\beta}\in \calh_{BC}^{0,s}(\mathfrak{g})$, and $\overline{\gamma}\in\calh_{BC}^{0,v}(\mathfrak{g})$, such that $$ \dbar\overline{\beta}\wedge\overline{\gamma}=0, \qquad \alpha\wedge\overline{\beta}\wedge\overline{\gamma}\neq 0.
$$
In this situation, $f_{\alpha\beta}=\alpha\wedge \overline{\beta}$ is a primitive of $\del\alpha\wedge \dbar\overline{\beta}$ and the following triple ABC-Massey product is well defined
\[
\langle[\del\alpha],[\dbar\overline{\beta}],[\overline{\gamma}]\rangle_{ABC}=[\alpha\wedge\overline{\beta}\wedge\overline{\gamma}]_{A}\in \frac{H_A^{p-1,s+v-1}(\mathfrak{g})}{[\overline{\gamma}]\cup H_A^{p-1,s-1}(\mathfrak{g})}.
\]
We then check that $[\alpha\wedge\overline{\beta}\wedge\overline{\gamma}]_A\neq 0$, by showing that $\alpha\wedge\overline{\beta}\wedge\overline{\gamma}$ is Aeppli harmonic. Finally, we prove that, as a triple Massey product, $[\alpha\wedge\overline{\beta}\wedge\overline{\gamma}]_A\neq 0$, by first proving that $\overline{\gamma}\wedge \ast(\alpha\wedge\overline{\beta}\wedge \overline{\gamma})$ is $\del\dbar$-exact. 

Indeed, let $\{\xi_j\}_{j=1}^{h_A^{p-1,s-1}}$ be a basis of invariant Aeppli harmonic $(p-1,s-1)$-forms and assume by contradiction that $[\alpha\wedge\overline{\beta}\wedge\overline{\gamma}]_A\in [\overline{\gamma}]\cup H_A^{p-1,s-1}(\mathfrak{g})$, i.e.,
\[
\alpha\wedge\overline{\beta}\wedge\overline{\gamma}=\sum_{j=1}^{h_A^{p-1,s-1}}
\lambda_j\overline{\gamma}\wedge \xi_j+\del R+\dbar S, \qquad \textstyle R\in \bigwedge^{p-2,s+v-1}(\mathfrak{g}), S\in\bigwedge^{p-1,s+v-2}(\mathfrak{g}). 
\]
Integrating this equality over $M$ against $\ast(\alpha\wedge\overline{\beta}\wedge \overline{\gamma})$, we obtain
\begin{equation}\label{eq:thm_ABC-M_prod}
0\neq \int_M ||\alpha\wedge\overline{\beta}\wedge\overline{\gamma}||^2\text{vol}=\sum_{j=1}^{h_A^{p-1,s-1}}\lambda_j\int_M\overline{\gamma}\wedge \xi_j\wedge\ast(\alpha\wedge\overline{\beta}\wedge\overline{\gamma}).
\end{equation}
Note that
$$
\overline{\gamma}\wedge\ast(\alpha\wedge\overline{\beta}\wedge\overline{\gamma})\wedge\xi_j=-\overline{\gamma}\wedge\ast(\alpha\wedge\overline{\beta}\wedge\overline{\gamma})\wedge\ast(\ast\xi_j)=-g(\overline{\gamma}\wedge\ast(\alpha\wedge\overline{\beta}\wedge\overline{\gamma}),\ast\xi_j)\text{vol}
$$ and 
\begin{align*} \ast\xi_j\in\calh_{BC}^{n-p+1,n-s+1}(\mathfrak{g}). 
\end{align*}
If 
\begin{align*}
\textstyle\overline{\gamma}\wedge\ast(\alpha\wedge\overline{\beta}\wedge\overline{\gamma})\in\im(\del\dbar|_{\bigwedge^{n-p+1,n-s+1}(\mathfrak{g})}),\quad
\end{align*}
then
\begin{align*}
\int_M g(\textstyle\overline{\gamma}\wedge\ast(\alpha\wedge\overline{\beta}\wedge\overline{\gamma}),\ast\xi_j)\text{vol}=\langle\langle \textstyle\overline{\gamma}\wedge\ast(\alpha\wedge\overline{\beta}\wedge\overline{\gamma}),\ast\xi_j\rangle\rangle_{L^2}=0,
\end{align*}
by the orthogonal decomposition \eqref{eq:decomp_hodge:_bott}. Hence, we obtain that the every summand of the last term of \eqref{eq:thm_ABC-M_prod} vanishes, which yields a contradiction. As a result, $[\alpha\wedge\overline{\beta}\wedge\overline{\gamma}]_A$ represents a non-vanishing triple ABC-Massey product on $\mathfrak{g}$.

Moreover, if $(M,J)=(M',J')\times(M'',J'')$ and either $(M',J')$ or $(M'',J'')$ admits a non-vanishing triple ABC-Massey product, so does $(M,J)$.

We refer to Appendix \ref{appendix} for the explicit description (up to signs and constants) of the triple ABC-Massey products for each case by listing structure equations, starting Bott-Chern cohomology classes and Aeppli cohomology representatives.
\end{proof}
\end{prop}
We then have immediately that
\begin{cor}
Let $(M,J)$ be a complex solvable manifold of either type $III.2-3$,  $IV.2-4,6$ or $V.2-10,12,15,17$. Then $(M,J)$ is not weakly formal.
\end{cor}
As a result, we are led to formulate the natural following
\begin{question*}
Does every complex solvable manifold admit non-vanishing triple ABC-Massey products?
\end{question*}

\section{Dolbeault formality of compact simply connected simple Lie groups of rank 2}\label{sec:7}
In this section, we study Dolbeault formality of even dimensional compact connected simply connected simple Lie groups of rank $2$ endowed with a left-invariant complex structure. Classically, it is well known that the only rank $2$ compact simply connected simple Lie groups are $A_2=\text{SU}(3)$, $B_2=\text{Spin}(5)$, and the exceptional Lie group $G_2$. They are, respectively, of real dimension $8$, $10$, and $14$ with non trivial Betti numbers
\[
b_3(\text{SU}(3))=b_5(\text{SU(3)})=1, \quad b_3(\text{Spin}(5))=b_7(\text{Spin}(5))=1, \quad b_3(G_2)=b_{11}(G_2)=1,
\]
up to the zeroth and the top Betti numbers being $1$.

In \cite{Sam}, it was shown that, given a compact even dimensional Lie group $G$, a left-invariant complex structure $J$ on $G$ can be given from the choice of a maximal torus $T$ of $G$ (a maximal compact Abelian subgroup of $G$), a complex structure on the Lie algebra of $T$, and a set of positive roots. Later, Pittie \cite{Pit} showed that all left-invariant complex structures on compact Lie groups arise in this way. We briefly recall the construction for semisimple Lie groups.

Let $G$ be a even dimensional compact (simply connected) semisimple Lie group, and let $\mathfrak{g}$ be its Lie algebra. 
Moreover, let $T$ be a maximal torus of $G$ and let $\mathfrak{t}$ be its Lie algebra. The action of $\mathfrak{t}$ on $\mathfrak{g}$ is semisimple, hence simultaneously diagonalizable over $\C$. As a consequence, we have the $\text{ad}_T$-invariant decomposition of $\mathfrak{g}^{\C}=\mathfrak{g}\otimes \C$ as
\[
\mathfrak{g}^{\C}=\mathfrak{t}^{\C}\oplus \sum_{\alpha \in R}\mathfrak{g}_{\alpha},
\]
where $R\subset (\mathfrak{t}_{\C})^*$ is the root system of $\mathfrak{g}_{\C}$ relative to $\mathfrak{t}_{\C}$ and
\[
\mathfrak{g}_{\alpha}:=\{X\in\mathfrak{g}^{\C}:\text{ad}_Y(X)=\alpha(Y)X, \quad \forall Y\in \mathfrak{t}\}, \qquad \alpha \in R.
\]
The subalgebra $\mathfrak{t}^{\C}$ is called the \emph{Cartan subalgebra} of $\mathfrak{g}^{\C}$ and it is unique, up to conjugation by an automorphisms.

Note that if $\alpha\in R$, then also $-\alpha\in R$ and $\mathfrak{g}_{-\alpha}=\overline{\mathfrak{g}_{\alpha}}$. A root $\alpha$ is then said \emph{positive} (and we will denote it by $\alpha>0$) if $i\alpha(H)>0$ for $H\in \mathfrak{t}$ such that $\beta(H)\neq 0$ for every $\beta\in R$. Therefore, we can decompose $\mathfrak{g}^{\C}$ as
\[
\mathfrak{g}^{\C}=\mathfrak{t}^{\C}\oplus \sum_{\alpha>0}\mathfrak{g}_{\alpha}\oplus \mathfrak{g}_{-\alpha}.
\]
A left-invariant complex structure $J$ on $G$ is then given by a complex subalgebra $\mathfrak{s}\subset \mathfrak{g}^{\C}$ satisfying $\mathfrak{s}\cap \mathfrak{g}=0$ and $\mathfrak{s}\oplus \overline{\mathfrak{s}}=\mathfrak{g}^{\C}$, i.e., the choice $\mathfrak{s}$ coincides with choosing the space of $(1,0)$-vectors on $\mathfrak{g}^{\C}$ with respect to the complex structure that is being defined. The subalgebra $\mathfrak{s}$ is said a \emph{Samelson subalgebra}, and it is always solvable.

Since $G$ is even dimensional, also $T$ is even dimensional, i.e., the rank of $G$ is even. Hence, a Samelson subalgebra is determined by a complex structure on $\mathfrak{t}$ by choosing $\mathfrak{a}\subset \mathfrak{t}^{\C}$ such that $\mathfrak{a}\cap \mathfrak{t}=0$, $\mathfrak{a}\oplus \overline{\mathfrak{a}}=\mathfrak{t}^{\C}$, and by choosing a set of positive roots $\alpha>0$, i.e., we can pick
\[
\mathfrak{s}=\mathfrak{a}\oplus \sum_{\alpha>0}\mathfrak{g}_{\alpha}.
\]
The left-invariant complex structure $J$ is said to be \emph{isotropic} if it is compatible with the {\em Killing form} of $G$, i.e., the non-degenerate symmetric bilinear form $B(X,Y):=\text{tr}(\text{ad}_X\circ \text{ad}_Y)$ on $\mathfrak{g}$. In particular, $-B$ is a positive scalar product on $G$ and if $G$ is  simple, then any bi-invariant metric on $G$ is a scalar multiple of it.

In this case, it was shown in \cite{Pit} that the Dolbeault cohomology of a compact simply connected simple Lie group of rank $2$ with isotropic left-invariant structure $(G,J)$ is given by the model
\begin{equation}\label{eq:dolb-cohom-rk-2}
H_{\dbar}^{\bullet,\bullet}(G,J)\cong \bigwedge\C\langle [x_{0,1}],[x_{2,1}]\rangle\otimes  \bigwedge\C\langle[y_{1,1}]\rangle/[y_{1,1}]^{n-1},
\end{equation}
where $[x_{0,1}]$ and  $[y_{1,1}]$ are elements of bidegree $(0,1)$ and $(1,1)$ respectively.

The following result was stated in \cite[§5]{Pit}. We provide here an explicit proof by making use of the Hodge diamond, which we derive from \eqref{eq:dolb-cohom-rk-2}.

\begin{thm}
Let $(G,J)$ be a compact connected  simply connected simple Lie group $G$ of rank 2, endowed with a isotropic invariant complex structure $J$. Then, the Fr\"olicher spectral sequence of $(G,J)$ degenerates at the second step.
\begin{proof}
We start by considering $G$ to be $SU(3)$, endowed with an isotropic left-invariant complex structure $J$. From the formula \eqref{eq:dolb-cohom-rk-2}, we obtain immediately that the Hodge diamond of $(SU(3),J)$ is given by
\begin{equation}\label{table-dolb-su(3)}
    \begin{tikzcd}
   & &  & & & 1 & & & & &\\
   & &  & & 0 &  & 1 & & & &\\
   & &  & 0 & & 1 & & 0 & & &\\
  & &   0 &  & 1 &  &1 &  & 0 & &\\
  & 0 &    & 0 & & 2 & & 0 & & 0 &\\
  & & 0  &  & 1 &  & 1 &  & 0 & & \\
   & &   & 0 & & 1 & & 0 & & &\\
   & &   &  & 1 &  & 0 &  & & &\\
  & &    &  & & 1 & &  & & &\\
    \end{tikzcd}
\end{equation}
We determine $E_2^{p,q}(G)$ from the following formula given in \cite[Theorem 3]{CFGU}
\begin{equation}\label{eq:second-page-cfgu}
E_2^{p,q}(G)=\frac{\ker( d_1\colon E_1^{p,q}(G)\rightarrow E_{1}^{p+1,q}(G)) }{\text{im}( d_1\colon E_r^{p-1,q}(G)\rightarrow E_{1}^{p,q}(G))}, \qquad d_1[\alpha^{p,q}]=[\del \alpha^{p,q}].
\end{equation}

As consequence, since $E_1^{p,q}(G)\cong H_{\dbar}^{p,q}(G)$, from \eqref{table-dolb-su(3)} and \eqref{eq:second-page-cfgu}, we obtain immediately that $E_1^{p,q}(G)= 0$ and $E_2^{p,q}(G)=0$ except for
\[
(p,q)\in\{(0,0),(0,1),(1,1),(2,1),(1,2),(2,2),(3,2),(2,3),(3,3),(4,3),(4,4)\}.
\]
In particular, we have that
\begin{gather*}
E_1^{0,0}(G)\cong \C, \quad E_1^{0,1}(G)\cong \C, \quad E_1^{1,1}(G)\cong \C, \quad E_1^{2,1}(G)\cong \C, \quad E_1^{1,2}(G)\cong \C,\\
E_1^{2,2}(G)\cong \C^2, \quad 
E_1^{3,2}(G)\cong \C, \quad E_1^{2,3}(G)\cong \C, \\
E_1^{3,3}(G)\cong \C, \quad 
E_1^{4,3}(G)\cong \C, \quad E_1^{4,4}(G)\cong \C.
\end{gather*}
We begin by considering $k=1$ and we prove that $E_2^{0,1}(G)=0$. By contradiction, let $[\alpha^{0,1}]\in E_2^{0,1}(G)$ such $0=d_2[\alpha^{0,1}]\in E_2^{2,0}(G)=0$. Since $|d_r|=(r,1-r)$, it can be easily seen that $d_r[\alpha^{0,1}]=0$, for every $r\geq 2$. In particular, $[\alpha^{0,1}]\in E_{\infty}^{0,1}(G)$. However, $0=b_1(G)= \dim E_{\infty}^{1,0}(G)+ \dim E_{\infty}^{0,1}(G)> 0$, yielding a contradiction. Hence,
\[
E_2^{0,1}(G)=0.
\]
We consider now $k=2$. We already know that $E_2^{2,0}(G)=E_2^{0,2}(G)=0$. We observe that
\[
0=E_2^{0,1}(G)=\ker d_1\colon E_1^{0,1}(G)\rightarrow E_1^{1,1}(G).
\]
As a consequence, $d_1\colon E_1^{0,1}(G)\rightarrow  E_1^{1,1}(G)$ is an isomorphism. Therefore,
\begin{equation}\label{eq:d_1(1,1)}
d_1(E_1^{0,1}(G))= E_1^{1,1}(G).
\end{equation}
Hence,
\[
E_2^{1,1}(G)\cong\frac{\ker (d_1\colon E_{1}^{1,1}(G)\rightarrow E_1^{2,1}(G))}{\text{im}\, (d_1\colon E_1^{0,1}(G)\rightarrow E_1^{1,1}(G))}=0.
\]
We continue with $k=4$. We consider
\[
E_2^{2,2}(G)=\frac{\ker (d_1\colon E_1^{2,2}(G)\rightarrow E_1^{3,2}(G))}{\text{im}\, (d_1\colon E_1^{1,2}(G)\rightarrow E_1^{2,2}(G))}.
\]
Let $0\neq [\alpha^{2,2}]\in E_2^{2,2}(G)$. Then, $d_2[\alpha^{2,2}]\in E_{2}^{4,1}(G)=0$ and, therefore,
\[
[\alpha^{2,2}]\in \ker (d_2\colon E_2^{2,2}(G)\rightarrow E_2^{4,1}(G))=\frac{\ker (d_2\colon E_2^{2,2}(G)\rightarrow E_2^{4,1}(G))}{\text{im}\, (d_2\colon E_2^{0,3}(G)\rightarrow E_2^{2,2}(G))}=E_3^{2,2}(G).
\]
It can then be seen that $d_r[\alpha^{2,2}]=0$, i.e., $0\neq [\alpha^{2,2}]\in E_{\infty}^{2,2}(G)$. On the hand,
\[
0=b_4(G)=\dim E_{\infty}^{2,2}(G)>0,
\]
which is a contradiction. Hence,
\[
E_{2}^{2,2}(G)=0.
\]
We go back to the case $k=3$. We prove that $E_2^{2,1}(G)\cong\C$ and $E_2^{1,2}=0$. We recall that
\[
E_2^{2,1}(G)\cong \frac{\ker (d_1\colon E_{1}^{2,1}(G)\rightarrow E_1^{3,1}(G))}{\text{im}(d_1\colon E_1^{1,1}(G)\rightarrow E_1^{2,1}(G))}
\]
However, from \eqref{eq:d_1(1,1)}, the map $d_1\colon E_1^{1,1}(G)\rightarrow E_1^{2,1}(G)$ is trivial. Moreover, $E_1^{3,1}(G)=0$, so that $d_1\colon E_1^{2,1}(G)\rightarrow E_1^{3,1}(G)$ is trivial and, consequently, $E_{2}^{2,1}(G)\cong \C$. Suppose now, by contradiction, that $E_2^{1,2}(G)\cong \C$, with
\[
E_2^{1,2}(G)\cong \ker (d_1\colon E_1^{1,2}(G)\rightarrow E_1^{2,2}(G)).
\]
Now, since
\[
E_{2}^{2,2}(G)\cong \frac{\ker (d_1\colon E_1^{2,2}(G)\rightarrow E_1^{3,2}(G))}{\text{im}(d_1\colon E_1^{1,2}(G)\rightarrow E_1^{2,2}(G))},
\]
we have that $E_2^{2,2}(G)\cong \ker (d_1\colon E_1^{2,2}(G)\rightarrow E_1^{3,2}(G))$. However, since $E_1^{2,2}(G)\cong \C^2$ and $E_1^{3,2}(G)\cong \C$, it must hold that
\[
\dim\left(\ker (d_1\colon E_1^{2,2}(G)\rightarrow E_1^{3,2}(G))\right)>0,
\]
hence $E_2^{2,2}(G)\neq 0$, which is a contradiction. Hence, $E_2^{1,2}(G)=0$ and, since $1=b_3(G)=\dim E_{\infty}^{1,2}(G)+\dim E_{\infty}^{2,1}(G)=\dim E_{\infty}^{2,1}(G)$, it holds that $$E_2^{2,1}(G)=E_{\infty}^{2,1}(G)\cong \C.$$
By Serre duality \cite[Corollary G]{Ste21}, we can recover that $E_2^{p,q}(G)=0$ except for $$E_2^{0,0}(G)\cong E_2^{2,1}(G)\cong  E_2^{2,3}(G)\cong E_2^{4,4}(G)\cong  \C$$ and the Fr\"olicher spectral sequence of $(G,J)$ degenerates at the second step.

For the cases of $G=\text{Spin(5)}$ and $G=G_2$ endowed with isotropic left-invariant complex structures, the degeneration of the Fr\"olicher spectral sequence at the second step can be proved with a similar argument. 
\end{proof}
\end{thm}

From a result of Tanré \cite[Théorème 9]{Tan94}, the Fr\"olicher spectral sequence of a compact connected Lie group $G$ with a left-invariant complex structure $J$ degenerates at the second step if and only if $(G,J)$ is Dolbeault formal. Hence, we obtain the following.
\begin{cor}\label{cor:7}
Compact connected simply connected simple Lie groups of rank 2 with left-invariant isotropic complex structures are Dolbeault formal.

\end{cor}

\section{Geometric Aeppli formality on meta-Abelian solvmanifolds}\label{sec:8}
In this section we prove that, under certain assumptions, complex splitting type solvmanifolds admit geometrically Aeppli formal metrics. As a warm up, we notice that geometrically Aeppli formal manifolds are closely related to Hermitian metrics satisfying {\em Hodge symmetry and decomposition}, i.e., metrics whose space of Dolbeault harmonic forms satisfy
\[
\overline{\mathcal{H}_{\dbar}^{p,q}(M)}\cong \mathcal{H}^{q,p}_{\dbar}(M), \qquad \bigoplus_{p+q=k}\mathcal{H}_{\dbar}^{p,q}(M)=\mathcal{H}_{dR}^k(M;\C).
\]
It is immediate to prove the following lemma, which provides a further characterization of geometrically Aeppli formal metrics.
\begin{lemma}\label{lemma:hodge-symm-geom-aeppl}
For a Hermitian metric $g$ on a compact complex manifold $(M,J)$, the following statements are equivalent:
\begin{enumerate}
    \item the metric $g$ is geometrically Aeppli formal,\\[2pt]
    \item the metric $g$ satisfies the Hodge symmetry and decomposition and it is geometrically Dolbeault formal.
\end{enumerate}
\begin{proof}
By Lemma \ref{lemma:geom-aeppli-form-def}, if the metric is geometrically Aeppli formal, all the bigraded spaces of complex harmonic forms coincide, hence the Dolbeault harmonic forms coincide in particular with the Bott-Chern harmonic forms, which are invariant under complex conjugation and closed under $\wedge$. Furthermore, the space of de Rham harmonic forms decomposes uniquely in terms of harmonic forms with respect to either Dolbeault, Bott-Chern, or Aeppli Laplacian. Viceversa, let $g$ satisfy the Hodge symmetry and decomposition. Then $(M,J)$ satisfies the $\del\dbar$-lemma and every $\alpha\in \mathcal{H}_{dR}^{k}(M;\C)$ decomposes uniquely as $\alpha=\sum_{p+q=k}\alpha^{p,q}$, with $\alpha^{p,q}\in \mathcal{H}_{\dbar}^{p,q}(M)$. Now, since $d\alpha=0$ and $\dbar\alpha^{p,q}=\dbar^*\alpha^{p,q}=0$, every form $\alpha^{p,q}$ is indeed $d$-closed and $d^*$-closed. As a consequence, every $\alpha^{p,q}$ is a harmonic representative for de Rham, Dolbeault, Bott-Chern, and Aeppli Laplacian altogether. By assumption, the metric $g$ is geometrically formal, hence, by Lemma \ref{lemma:geom-aeppli-form-def}, the metric is geometrically Aeppli formal.
\end{proof}
\end{lemma}
We now recall the definition of complex splitting type solvmanifolds.
\begin{defi}[\cite{Kas}]
Let $G$ be a connected simply connected Lie group given by the semidirect product $G=\C^n\ltimes N$ with a left-invariant complex structure $J=J_{\C^n}\ltimes J_N$ such that
\begin{enumerate}
    \item  $N$ is a simply connected nilpotent Lie group with a left-invariant complex structure $J_N$. Let $\mathfrak{a}$ and $\mathfrak{n}$ be the Lie algebras of $\C^n$ and $N$ respectively.\\[2pt]
    \item For any $t\in\C^n$, $\phi(t)$ is a holomorphic automorphism of $(N,J_N)$.\\[2pt]
    \item $\phi$ induces a semi-simple action on the Lie algebra of $N$.\\[2pt]
    \item $G$ has a lattice $\Gamma$. (Then $\Gamma$ can be written by $\Gamma=\Gamma'\ltimes_{\phi} \Gamma''$ such that $\Gamma'$ and $\Gamma''$ are lattices of $\C^n$ and $N$ respectively
and for any $t\in\Gamma'$ the action $\phi(t)$ preserves $\Gamma''$.)\\[2pt]
\item  The inclusion $\bigwedge^{\bullet,\bullet}\mathfrak{n}^{\ast}_\C \hookrightarrow \mathcal{A}^{\bullet,\bullet}(\Gamma''\backslash N)$ induces an isomorphism $H_{\dbar}^{\bullet,\bullet}(\mathfrak{n})\cong H_{\dbar}(\Gamma''\backslash N)$ 
where $\bigwedge^{\bullet,\bullet}
\mathfrak{n}^*_{\C}$ is the differential bigraded algebra of the complex valued left-invariant differential forms on the nilmanifold $\Gamma''\backslash N$.
\end{enumerate}
Then, the solvmanifold $(\Gamma\backslash G,J)$ is said to be of {\em complex splitting type}.
\end{defi}
The complex structure $J_N$ induces the decomposition $\mathfrak{n}\otimes \C=\mathfrak{n}^{1,0}\oplus \mathfrak{n}^{0,1}$, which is preserved by  $\phi(t)$, by (3). Moreover, the induced action of $\phi(t)$ on $(\mathfrak{n}^*)^{1,0}$ is semisimple. This also implies that, hence, there exist characters $\alpha_1(t), \dots, \alpha_m(t)$ of $\C^n$ such that $\{x_1,\dots, x_n,\alpha_1^{-1}y_1,\dots,\alpha_m^{-1}y_m\}$ is a basis of $(\mathfrak{g}^\ast)^{1,0}$, where $\{y_1,\dots, y_m\}$ is $N$-left-invariant basis of $(\mathfrak{n}^\ast)^{1,0}$ so that the induced action on $(\mathfrak{n}^\ast)^{1,0}$ is expressed as 
$$\phi(t)=\text{diag}(\alpha_1^{-1}(t),\dots, \alpha_m^{-1}(t)).$$ 
By \cite{Kas}, for every $\alpha_j$ and $\overline{\alpha_j}$, there exists unique unitary characters $\beta_j$ and $\gamma_j$ such that $\alpha_j\beta_j^{-1}$ and $\overline{\alpha}_j\gamma_j^{-1}$ are holomorphic characters.\newline 
For any multi-index $I:=1\leq i_1<\dots< i_p\leq n+m)$, denote $\alpha^I:=\alpha^{i_1}\cdot \cdot \alpha^{i_p}$ and $|I|=p$. Then, the following theorem allows us to determine a finite dimensional subcomplex of $\mathcal{A}^{\bullet,\bullet}(\Gamma\backslash G)$ which compute the Bott-Chern cohomology of $(\Gamma\backslash G,J)$.
\begin{thm}[\cite{AK17}]
Let $(\Gamma\backslash G,J)$ be a complex splitting type solvmanifold. Set
\[
B_{\Gamma}^{p,q}:=\left\langle x_I\wedge \left(\alpha_J^{-1}\beta_J\right)y_J\wedge \overline{x}_K\wedge \left(\overline{\alpha}_L^{-1}\gamma_L\right)\overline{y}_{L}: |I|+|J|=p, |K|+|L|=q,\,\, \beta_J\gamma_L|_{\Gamma}\equiv 1  \right\rangle
\]
and $\overline{B}_{\Gamma}^{p,q}:=\overline{B_{\Gamma}^{q,p}}$. Let $C_{\Gamma}^{\bullet,\bullet}:=B_{\Gamma}^{\bullet,\bullet}+\overline{B}_{\Gamma}^{\bullet,\bullet}$. Then, the inclusion $C_{\Gamma}^{\bullet,\bullet}\hookrightarrow \mathcal{A}^{\bullet,\bullet}(\Gamma\backslash G)$ induces the isomorphisms
\[
H(C_{\Gamma}^{p-1,q-1}\xrightarrow{\del\dbar} C_{\Gamma}^{p,q}\xrightarrow{d} C_{\Gamma}^{p+1,q}\oplus C_{\Gamma}^{p,q+1})\cong H_{BC}^{p,q}(\Gamma\backslash G).
\]
\end{thm}
Let us now focus on the more restrictive hypotheses that $G=\C^n\ltimes_{\phi}\C^m$, and the action of $\phi$ is semisimple. This family of complex splitting type solvmanifolds is called \emph{meta-Abelian}.
\begin{rmk}
    It was proved in \cite{Kasuya-Kokyu} and \cite{Kasuya-Hodge-sym-dec} that such manifolds are always, respectively, geometrically formal and geometrically Dolbeault formal.
\end{rmk}

In this setting, the space of left-invariant $(p,q)$-forms on $G$ admits the natural basis
\[
\left\{ dz_I\wedge \alpha_J^{-1} dw_J\wedge d\overline{z}_K\wedge \overline{\alpha}_L^{-1}d\overline{w}_L: |I|+|J|=p, |K|+|L|=q\right\},
\]
with $z_1,\dots, z_n$ and $w_1,\dots, w_m$  holomorphic coordinates for $\C^n$, respectively, $\C^m$. It was proved in \cite[Lemma 4.6]{Rub} that such a  $(\Gamma\backslash G,J)$ satisfies the $\del\dbar$-lemma if and only if the complexes $B_{\Gamma}^{\bullet,\bullet}$ and $\overline{B}_{\Gamma}^{\bullet,\bullet}$ coincide (and $d|_{B^{\bullet,\bullet}}\equiv 0$).  Moreover, it was observed that $X$ admits a geometrically Bott-Chern formal metric. We prove the following.
\begin{thm}\label{thm:geom-aeppli-metaabelian}
Let $X:=(\Gamma\backslash G,J)$ be a meta-Abelian complex splitting type solvmanifold. Then, $X$ satisfies the $\del\dbar$-lemma if and only if it admits a geometrically Aeppli formal metric.
\begin{proof}
If $X$ admits a geometrically formal Aeppli metric, then it satisfies the $\del\dbar$-lemma, by Lemma \ref{lemma:geom-aeppli-form-def}.

Viceversa, assume that $X$ satisfies the $\del\dbar$-lemma. Then the metric 
$$g=\sum_{i=1}^ndz_idz_i +\sum_{j=1}^m\alpha_j^{-1}\overline{\alpha}_j^{-1}dw_jd\overline{w}_j$$ is a Hermitian metric on $X$ for which $\ast_gB_{\Gamma}^{\bullet,\bullet}\subset B_{\Gamma}^{\bullet,\bullet}$. Now, by \cite[Lemma 4.6]{Rub}, we have that $d|_{B_{\Gamma}^{\bullet,\bullet}}\equiv 0$. Hence, the spaces $C_{\Gamma}^{\bullet,\bullet}=B_{\Gamma}^{\bullet,\bullet}=\overline{B}_{\Gamma}^{\bullet,\bullet}$ coincide with the space of harmonic forms with respect to either Dolbeault, Bott-Chern, or Aeppli Laplacians, and have a structure of algebra. Hence, the metric $g$ is geometrically Aeppli formal, by Lemma \ref{lemma:geom-aeppli-form-def}.
\end{proof}
\end{thm}
In \cite{Kasuya-Hodge-sym-dec}, sufficient conditions were given so that a meta-Abelian complex splitting solvmanifold satisfies the $\del\dbar$-lemma. Indeed, following \cite[Condition 3.3]{Kasuya-Hodge-sym-dec}, let us assume that for each multi-indices $J$,$L\subset \{1,\dots,m\}$,
\begin{equation}\label{eq:kasuya}
(\beta_J\gamma_L)|_{\Gamma} =1\iff \alpha_J \overline{\alpha}_L \equiv \underline{1}
\end{equation}
where $\underline{1}$ is the trivial character of $\C^n$. Then, it is readily seen that if such a condition holds, then $B_{\Gamma}^{\bullet,\bullet}=\overline{B}_{\Gamma}^{\bullet,\bullet}$, hence it satisfies the $\del\dbar$-lemma by \cite[Corollary 4.7]{Rub}. Then, by Theorem \ref{thm:geom-aeppli-metaabelian}, $X$ admits a geometrically formal Aeppli metric. Alternatively, in  \cite[Theorem 3.2, Corollary 3.2]{Kasuya-Hodge-sym-dec}, it is proved that the invariant metric $$g=\sum_{i=1}^ndz_id\overline{z}_i+\sum_{j=1}^m\alpha_j^{-1}\overline{\alpha}_j^{-1}dw_jd\overline{w}_j$$ satisfies Hodge symmetry and decomposition and the space of Dolbeault harmonic forms has a structure of algebra. Then, Lemma \ref{lemma:hodge-symm-geom-aeppl} applies. To summarize, we have obtained the following.
\begin{cor}
The meta-Abelian complex splitting type solvmanifolds satisfying condition \eqref{eq:kasuya} are geometrically Aeppli formal.
\end{cor}
\noindent {\bf Generalized Nakamura manifolds.}
Let us revisit the construction of \textit{generalized Nakamura manifolds} as presented in \cite{CT}. Under certain additional hypotheses, we demonstrate that such manifolds admit invariant geometrically Aeppli metrics. Let $M\in \text{SL}(n,\Z)$  be a matrix with positive eigenvalues and diagonalizable over $\R$, so that there exists $P\in \text{GL}(n,\R)$ and numbers $\lambda_1, \dots, \lambda_n\in\R$ such that
\[
PMP^{-1}=\text{diag}(e^{\lambda_1},\dots, e^{\lambda_n})=:\Lambda.
\]
Define the group homomorphism $\rho\colon \C\rightarrow \text{GL}(n,\C)$ by
\[
\rho(w)=\text{diag}(e^{\frac{1}{2}\lambda_1(w+\overline{w})},\dots,e^{ \frac{1}{2}\lambda_n(w+\overline{w})}),
\]
so that we have a Lie group $G_M:=\C\ltimes_{\rho} \C^n$. By fixing a $\tau\in \R\setminus\{0\}$ and letting
\[
\Gamma_{\tau}':=\Z\oplus \sqrt{-1}\tau \cdot \Z\subset \C, \qquad \Gamma_{P}'':=P\cdot \Z^n\oplus \sqrt{-1}P\cdot \Z^n\subset \C^n,
\]
we have the lattice $\Gamma_{P,\tau}:=\Gamma_{\tau}'\ltimes_{\rho}\Gamma_P''$ of $G_M$. We denote by  $N_{M,P,\tau}:=\Gamma_{P,\tau}\backslash G_M$ the resulting compact complex manifold, which is endowed with the natural invariant complex structure given by the choice of left-invariant $(1,0)$-forms on $G_M$
\[
\phi^0=dw, \quad \phi^i=e^{-\frac{1}{2}\lambda_i(w+\overline{w})}dz_i, \quad i\in\{1,\dots, n\}.
\]
Let us fix the standard metric $\omega=\frac{i}{2}\sum_{i=0}^n\phi^i\wedge \overline{\phi}^i$.

For every pair of multi-indices $I=1\leq i_1<\dots <i_p\leq n$ and $ J=1\leq j_1<\dots <j_q\leq n$,  set  $c_{IJ}:=\sum_{\sigma=1}^p\lambda_{i_{\sigma}}+\sum_{\tau=1}^q\lambda_{j_{\tau}}$ and $f_{IJ}:=e^{-\frac{1}{2}c_{IJ}(w-\overline{w})}$. Then, it is shown in \cite[Section 4.4]{CT}, that the complex 
\[
B_{\Gamma_{P,\tau}}^{p,q}:=\left\{\begin{cases}f_{IJ}\phi^{I}\wedge\phi^{\overline{J}}, \quad |I|=p, \quad |J|=q,\\
f_{IJ}\phi^0\wedge \phi^{I}\wedge \phi^{\overline{J}}, \quad |I|=p-1, \quad |J|=q,\\
f_{IJ}\phi^{\overline{0}}\wedge \phi^{I}\wedge \phi^{\overline{J}}, \quad |I|=p, \quad |J|=q-1,\\
f_{IJ}\phi^0\wedge\phi^{\overline{0}}\wedge \phi^{I}\wedge \phi^{\overline{J}}, \quad |I|=p-1, \quad |J|=q-1
\end{cases}: \tau\cdot c_{IJ}\in 2\pi\Z \right\}
\]
satisfies $ B_{\Gamma_{P,\tau}}^{\bullet,\bullet}\cong H_{\dbar}^{\bullet,\bullet}(X)$ and it is invariant under the Hodge $\ast$-operator with respect to $\omega$.
Moreover, $\tau\in\R\setminus\{0\}$ can be chosen such that for every pair of the multi-indices $I$ and $J$, satisfying $\tau\cdot c_{IJ}\in 2\pi\Z$ is equivalent to  $c_{IJ}=0$. As a consequence, $f_{IJ}\equiv 1$. With this choice of $\tau$, the manifold $X$ satisfies the $\del\dbar$-lemma \cite[Theorem 4.14]{CT}. Then, by Theorem \ref{thm:geom-aeppli-metaabelian}, the metric $\omega$ is geometrically Aeppli formal on $X$.

\section{Geometric Dolbeault and Bott-Chern formality of Calabi-Eckmann manifolds}\label{sec:CE}
Let $M_{u,v}$ be a Calabi-Eckmann manifold of complex dimension $u+v+1$, with $0\leq u\leq v$. In particular, $M_{u,v}$ is diffeomorphic to $S^{2u+1}\times S^{2v+1}$ and it is the total space of the holomorphic fibration
\[
\begin{tikzcd}
S^1\times S^1\ar[r] & M_{u,v}\ar[d,"\pi"]\\
& \mathbb{CP}^u\times \mathbb{CP}^v,
\end{tikzcd}
\]
where each factor $S^1$ is the fiber of the respective $S^1$-fibration over the projective space. By construction, see e.g., \cite{Kob63}, there exist two connection $1$-forms $\eta_u$ and $\eta_v$ along the fibers such that $d\eta_u=\pi^*\omega_{FS,u}$ and $d\eta_v=\pi^*\omega_{FS,v}$. The Calabi-Eckmann complex structure $J$ on \begin{equation}\label{eq:splitting_CE}
T^*M_{u,v}\cong T^*(\pi^*(\mathbb{CP}^u\times \mathbb{CP}^v))\oplus T^*(S^1\times S^1)
\end{equation}
is just the pullback of the product complex structure on $T^*(\mathbb{CP}^u\times \mathbb{CP}^v)$ and it acts on the vertical components as $J\eta_v=\eta_u$. We can then set $\phi:=\eta_u-i\eta_v\in\mathcal{A}^{1,0}(M_{u,v})$ and, for the sake of brevity, we set $\pi^*\omega_{FS,u}=:\omega_1$ and $\pi^*\omega_{FS,v}=:\omega_2$. The fundamental form $\Omega$ of the Hermitian metric $g$ on $M_{u,v}$ for which the decomposition \eqref{eq:splitting_CE} is orthogonal and $\phi$ has norm $1$, is given by
\begin{equation*}
\Omega:=\omega_1+\omega_2+\frac{i}{2}\phi\wedge\overline{\phi}.
\end{equation*}
The volume form associated to $\Omega$ is then
\[
\frac{\Omega^{u+v+1}}{(u+v+1)!}=\frac{i}{2u!v!}\omega_1^u\wedge\omega_2^v\wedge\phi\wedge\overline{\phi}.
\]
We compute
\[
d\phi=d(\eta_u-i\eta_v)=\omega_1-i\omega_2\in\mathcal{A}^{1,1}(M_{u,v}),
\]
so that, in particular, $\del\phi=0$ and $d\phi=\dbar \phi$. Furthermore
\begin{equation}\label{eq:ddc_CE_phiphibar}
\del\dbar(\phi\wedge\overline
{\phi})=\dbar\phi\wedge\del\overline{\phi}=(\omega_1-i\omega_2)\wedge(\omega_1+i\omega_2)=\omega_1^2+\omega_2^2,
\end{equation}
which vanishes for $(u,v)\in\{(0,0),(0,1),(1,1)\}$, corresponding to the Calabi-Eckmann manifolds diffeomorphic to, respectively, $S^1\times S^1$, $S^1\times S^3$, and $S^3\times S^3$.

In these cases, it is known that there exist geometrically Bott-Chern formal metrics (see, e.g., \cite{AT15,TT17}) and these manifolds are weakly formal \cite{MS24}. More precisely, the Bott-Chern cohomology spaces of $M_{0,1}$ and $M_{1,1}$ are given in our terminology by, respectively, 

\begin{align*}
H_{BC}^{1,1}(M_{0,1})&=\C\langle[\omega_2]\rangle\\
H_{BC}^{2,1}(M_{0,1})&=\C\langle[\omega_2\wedge \phi]\rangle\\
H_{BC}^{1,2}(M_{0,1})&=\C\langle[\omega_2\wedge\overline{\phi}]\rangle\\
H_{BC}^{2,2}(M_{0,1})&=\C\langle[\omega_2\wedge\phi\wedge\overline{\phi}]\rangle,
\end{align*}

visually represented as 
\begin{equation*}
\begin{tikzcd}[sep=small]
 &&1&&\\
 &0&&0&\\
0&&\omega_2&&0\\
&\omega_2\phi&&\omega_2\overline{\phi}&\\
&&\omega_2\phi\overline{\phi}&&
\end{tikzcd}
\end{equation*}

and

\begin{align*}
H_{BC}^{1,1}(M_{1,1})&=\C \langle[\omega_1],[\omega_2]\rangle\\
H_{BC}^{2,1}(M_{1,1})&=\C \langle[(\omega_1+i\omega_2)\wedge \phi]\rangle\\
H_{BC}^{1,2}(M_{1,1})&=\C \langle[(\omega_1-i\omega_2)\wedge\overline{\phi}]\rangle\\
H_{BC}^{2,2}(M_{1,1})&=\C \langle[\omega_1\wedge\omega_2]\rangle\\
H_{BC}^{3,2}(M_{1,1})&=\C \langle[\omega_1\wedge\omega_2\wedge \phi]\rangle\\
H_{BC}^{2,3}(M_{1,1})&=\C \langle[\omega_1\wedge\omega_2\wedge\overline{\phi}]\rangle\\
H_{BC}^{3,3}(M_{1,1})&=\C\langle[\omega_1\wedge\omega_2\wedge\phi\wedge\overline{\phi}]\rangle,
\end{align*}

visually represented as
\begin{equation*}
\begin{tikzcd}[sep=small]
 &&&1&&&\\
 &&0&&0&&\\
&0&&\omega_1,\omega_2&&0&\\
0&&(\omega_1+i\omega_2)\phi&&(\omega_1-i\omega_2)\overline{\phi}&&0\\
&0&&\omega_1\omega_2&&0&\\
&&\omega_1\omega_2\phi&&\omega_1\omega_2\overline{\phi}&&\\
&&&\omega_1\omega_2\phi\overline{\phi}&&&
\end{tikzcd}
\end{equation*}
\vspace{0.2cm}

where we have listed the Bott-Chern harmonic representative of each cohomology class with respect to the standard metric $\Omega$. By a straightforward computation, it is immediate to see that both $(\calh_{BC}^{\bullet,\bullet}(M_{0,1}),\wedge)$ and $(\calh_{BC}^{\bullet,\bullet}(M_{1,1}),\wedge)$ are algebras, i.e., the standard metric $\Omega$ is geometrically Bott-Chern formal.

We focus then on the other cases, i.e., from now on we assume that $v\geq 2$.

Beforehand, we prove the following useful lemma, which ensures that there exists a model for computing the Bott-Chern cohomology of every Calabi-Eckmann manifold $M_{u,v}$, $u\leq v$.

\begin{lemma}\label{lem:model-BC}
Let $M_{u,v}$ be a Calabi-Eckmann manifold. Then, the injection of algebras
\begin{align*}
\cala^{\bullet,\bullet}&:=\textstyle\bigwedge^{\bullet,\bullet}\C\langle\phi,\overline{\phi},\omega_1,\omega_2\rangle \hookrightarrow \textstyle\bigwedge^{\bullet,\bullet}(M_{u,v})
\end{align*}
with
\begin{gather*}
|\phi|=(1,0),\quad |\overline{\phi}|=(0,1),\quad |\omega_1|=(1,1),\quad  |\omega_2|=(1,1),\\
\dbar\phi=\omega_1-i\omega_2,\quad \del\overline{\phi}=\omega_1+i\omega_2,
\end{gather*}
is a weak-equivalence.
\begin{proof}
From, e.g., \cite{MS24}, if the $E_1$-isomorphism given by a model for the Dolbeault cohomology of a compact complex manifold preserves the real structure, then it induces a weak equivalence, i.e., an isomorphism also at the level of Bott-Chern and Aeppli cohomology.
Hence, it suffices to show that there exists a weak equivalence between $\cala^{\bullet,\bullet}$ and the model of Dolbeault cohomology of any Calabi-Eckmann manifold $M_{u,v}$, which is given \cite{Tan94} by
\begin{align*}
(\textstyle\bigwedge^{\bullet,\bullet}\langle a,b,a',b',\alpha,\beta\rangle,\del,\dbar),
\end{align*}
where
\[
|a|=|a'|=(1,1),\, |b|=(u+1,u),\,|b'|=(v+1,v),\,|\alpha|=(1,0), \,|\beta|=(0,1),
\]
and
\begin{align*}
&\dbar a=\del a=\dbar a'=\del a'=0, \quad \dbar b=a^{u+1},\quad  \dbar b'=a'^{v+1},\\
&\del\alpha=0, \quad \dbar\alpha=a-ia',  \quad \del\beta=a+ia',  \quad \dbar\beta=0.
\end{align*}
However, it is immediate to see that the map defined by
\[
a\mapsto\omega_1, \quad b\mapsto 0, \quad a'\mapsto \omega_2, \quad b'\mapsto0, \quad \alpha\mapsto \phi, \quad \beta\mapsto \overline{\phi}
\]
gives the desired weak equivalence.
\end{proof}
\end{lemma}
\begin{rmk}
The proof of Lemma \ref{lem:model-BC} applies similarly to every model \cite{Tan94} of Dolbeault for principal holomorphic torus fiber bundles.
\end{rmk}
Thanks to Lemma \ref{lem:model-BC}, we can complete the description \cite{Ste21} of the Bott-Chern numbers of any Calabi-Eckmann manifold, by studying the remaining case, i.e., $u=v$.
\begin{lemma}\label{lemma:BC-numbers-CE}
Let $M_{u,u}$ be the Calabi Eckmann manifold of diffeomorphism type $S^{2u+1}\times S^{2u+1}$. Then
\begin{equation}
    h_{BC}^{p,q}(M_{u,u})=\begin{cases} 
    2, \quad \text{if} \quad u\geq 1, \quad (p,q)\in\{(1,1),\dots, (u,u)\}\\
    1, \quad \text{if} \quad (p,q)\in\{(0,0),(2u+1,2u+1)\}\\
    \quad \,\,\,\,\, \text{or}\quad  (p,q)\in\{(u,u+1),\dots,(2u,2u+1)\}\\
    \quad \,\,\,\,\, \text{or} \quad (p,q)\in \{(u+1,u),\dots,(2u+1,2u)\}\\
    \quad \,\,\,\,\,\text{or}\quad  u\equiv_2 1, \,\,(p,q)=(u+1,u+1)\\
    0, \quad \text{else}.
    \end{cases}
\end{equation}
\begin{proof}
By straightforward computations via the model of Lemma \ref{lem:model-BC}, it is easy to check the assertions. We show the computation for $h_{BC}^{u+1,u+1}(M_{u,u})$. Let us consider an arbitrary $(u+1,u+1)$-form in $\mathcal{A}^{u+1,u+1}$, i.e., 
\[
\alpha:=\sum_{j=1}^{u}a_j\omega_1^{j}\wedge\omega_2^{u-j+1}+\sum_{j=0}^{u}b_j\phi\wedge\overline{\phi}\wedge\omega_1^{j}\wedge\omega_2^{u-j}, \qquad a_j,b_j\in\C.
\]
This form is $d$-closed if and only if $b_j=0$, for $j=0,\dots, u$, i.e., if
\[
\alpha=\sum_{j=1}^{u}a_j\omega_1^{j}\wedge \omega_2^{u-j+1}.
\]
Let us denote by $V:=\C\langle\omega_1^{j}\wedge \omega_2^{u-j+1}:j\in\{1,\dots, u\}\rangle$ the kernel of $d$ on $\bigwedge^{u+1,u+1}\langle\phi,\overline{\phi},\omega_1,\omega_2\rangle$. Note that $\dim_\C V=u$. On the other hand, let $\beta$ be any element in $\bigwedge^{u,u}\langle\phi,\overline{\phi},\omega_1,\omega_2\rangle$, i.e.,
\[
\beta=\sum_{j=0}^u e_j \omega_{1}^{j}\wedge \omega_2^{u-j}+\sum_{j=0}^{u-1}f_j\phi\wedge \overline{\phi}\wedge \omega_1^{j}\wedge\omega_2^{u-j-1}, \qquad e_j,f_j\in\C.
\]
Then,
\begin{align*}
\del\dbar\beta&=\del\dbar \left(\sum_{j=0}^{u-1}f_j\phi\wedge \overline{\phi}\wedge \omega_1^{j}\wedge\omega_2^{u-j-1}\right)
\\&=\sum_{j=0}^{u-1}f_j(\omega_1^2+\omega_2^2)\wedge \omega_1^j\wedge \omega_2^{u-j-1}\\
&=\sum_{j=0}^{u-1}f_j(\omega_1^{j+2}\wedge \omega_2^{u-j-1}+\omega_1^j\wedge \omega_2^{u-j+1}).
\end{align*}
Let us set $W:=\C\langle\phi\wedge \overline{\phi}\wedge \omega_1^j\wedge \omega_2^{u-j-1}: j=0, \dots, u-1\rangle$, $\dim_\C W=u$. Then, the operator $\del\dbar\colon W\rightarrow V$ can be represented with respect to the bases above as a $u\times u$ matrix $A$ with $0$'s everywhere except the superdiagonal and the subdiagonal. Now, the matrix $A$ has rank $u$ if $u\equiv_2 0$ and has rank $u-1$ if $u\equiv_2 1$. As a consequence,
\[
\dim_\C H_{BC}^{u+1,u+1}(\mathcal{A}^{\bullet,\bullet})=\dim_\C \ker d\cap \mathcal{A}^{u+1,u+1}-\dim_\C \del\dbar(\mathcal{A}^{u,u})=\begin{cases}
    1, \quad u\equiv_2 1\\
    0, \quad u\equiv_2 0,
\end{cases}
\]
proving the lemma.
\end{proof}
\end{lemma}
We can write the explicit expression of the Bott-Chern harmonic forms at the bottom and top of the Bott-Chern diamond.
\begin{lemma}\label{lemma:3}
The forms $\omega_1\in\cala^{1,1}(M_{u,v})$,  $\omega_2\in\cala^{1,1}(M_{u,v})$, $\omega_1\wedge\omega_2\in\cala^{2,2}(M_{u,v})$, $\omega_1^u\wedge\omega_2^v\wedge \phi\in\cala^{u+v+1,u+v}(M_{u,v})$, and $\omega_1^u\wedge\omega_2^v\wedge\overline{\phi}\in\cala^{u+v,u+v+1}(M_{u,v})$ are Bott-Chern harmonic with respect to $\Omega$.
\begin{proof}
The thesis follows from straightforward computations using the model of Bott-Chern cohomology of Lemma \ref{lem:model-BC}.
\end{proof}
\end{lemma}
\begin{rmk}
From the computation of Bott-Chern numbers of Calabi-Eckmann manifolds of type $M_{0,v}$ \cite{Ste21}, we have that
\[
h_{BC}^{p,q}(M_{0,v})=\begin{cases}
    1,\, \text{if}\, \,\,\,(p,q)\in\{(0,0),(1,1),(v,v+1),(v+1,v),(v+1,v+1)\},\\
    0, \,\,\text{otherwise}.
\end{cases}
\]
In particular, the forms $\omega_2$, $\omega_2^v\wedge \phi$, $\omega_2^v\wedge \overline{\phi}$, and $\omega_2^v\wedge\phi\wedge\overline{\phi}$ determine the whole space of Bott-Chern harmonic forms on $M_{0,v}$ with respect to the standard metric $\Omega$. 
\end{rmk}
\begin{cor}\label{cor:no-geom-BC-std-CE}
Let $M_{u,v}$ be a Calabi-Eckmann manifold with $0\leq u\leq v$ and $v\geq 2$. Then, the standard metric $\Omega:=\omega_1+\omega_2+\frac{i}{2}\phi\wedge\overline{\phi}$ is not geometrically Bott-Chern formal.
\begin{proof}
Suppose that $g$ is geometrically Bott-Chern harmonic. Then, $\omega_1^2$, $\omega_2^2$, and $\omega_1^2+\omega_2^2$  are Bott-Chern harmonic. Note that, since $v\geq2$, $\omega_2^2\neq 0$ and, as a consequence, $$\omega_1^2+\omega_2^2\neq 0.$$ However, by equation \eqref{eq:ddc_CE_phiphibar}, $$\del\dbar(\phi\wedge\overline{\phi})=\omega_1^{2}+\omega_2^2,$$ which is a contradiction with decomposition \eqref{eq:decomp_hodge:_bott}. Thus, $g$ is not geometrically Bott-Chern formal.
\end{proof}
\end{cor}
We turn now to the study of geometrically Dolbeault formal metrics on Calabi-Eckmann manifolds. We recall that the Hodge numbers of a Calabi-Eckmann manifold $M_{u,v}$ are given \cite{Bor} by
\begin{equation}
    h_{\dbar}^{p,q}(M_{u,v})=\begin{cases}
        1, \,\,\text{if} \,\,p\leq u, \,\, q\in\{p,p+1\},\\
        \quad \,\text{or}\,\, p>v, \,\, q\in\{p,p-1\},\\
        0, \,\, \text{else},
    \end{cases}
\end{equation}
and by \cite{NT78,MS24}, Calabi-Eckmann manifolds are Dolbeault formal but not strictly formal.
\begin{cor}\label{lemma:dolb-geom-CE-std}
Let $M_{u,v}$ be a Calabi-Eckmann manifold. Then the standard metric $\Omega:=\frac{i}{2}\phi\wedge\overline{\phi}+\omega_1+\omega_2$ is geometrically Dolbeault formal if and only if $u=0$.
\begin{proof}
If $u=0$, then the Dolbeault harmonic forms with respect to $\Omega$ are
\begin{gather*}
\mathcal{H}_{\dbar}^{0,0}(M_{0,v})=\C\langle1\rangle, \quad \mathcal{H}_{\dbar}^{0,1}(M_{0,v})=\C\langle\overline{\phi}\rangle, \quad \mathcal{H}_{\dbar}^{v+1,v}(M_{0,v})=\C\langle\phi\wedge\omega_2^v\rangle, \\
\mathcal{H}_{\dbar}^{v+1,v+1}(M_{0,v})=\C\langle\phi\wedge\overline{
\phi}\wedge\omega_2^v\rangle, 
\end{gather*}
and clearly $(\mathcal{H}_{\dbar}^{\bullet,\bullet}(M_{0,v}),\wedge)$ is an algebra.

Vice versa, if $u\neq 0$, we have two cases: $u=1$, and $u\geq 2$.\\
If $u=1$, then it is easy to see that $$\mathcal{H}_{\dbar}^{1,1}(M_{1,v})=\C\langle v\omega_1+i\omega_2\rangle.$$ Suppose by contradiction that $\Omega$ is geometrically Dolbeault formal. Then
$$
(v\omega_1+i\omega_2)^2=v^2\omega_1^2-\omega_2^2+2iv\omega_1\wedge\omega_2=-\omega_2^2+2iv\omega_1\wedge\omega_2\neq 0
$$
is Dolbeault harmonic and non-zero. However, it can be checked that it is also $\dbar$-exact, since 
$$\dbar(\phi\wedge(-(2v+1)\omega_1-i\omega_2))=(v\omega_1+i\omega_2)^2.$$ This yields a contradiction and, as a result, $\Omega$ is not geometrically Dolbeault formal.

If $u\geq 2$, we can check that
\[
\mathcal{H}_{\dbar}^{1,1}(M_{u,v})=\C\langle v\omega_1+iu\omega_2\rangle, \quad 
\mathcal{H}_{\dbar}^{2,2}(M_{u,v})=\C\langle
\omega_1\wedge\omega_2-\frac{iv}{2(u-1)}\omega_1^2
+\frac{iu}{2(v-1)}\omega_2^2\rangle.
\]
On the other hand,  $$0\neq (v\omega^1+iu\omega_2)^2=v^2\omega_1^2-u^2\omega_2^2+2iuv\omega_1\wedge \omega_2\notin \C\langle\omega_1\wedge\omega_2-\frac{iv}{2(u-1)}\omega_1^2
+\frac{iu}{2(v-1)}\omega_2^2\rangle,$$
hence $\Omega$ is not geometrically Dolbeault formal.
\end{proof}
\end{cor}

\begin{rmk}
By Corollary \ref{cor:no-geom-BC-std-CE} and Theorem \ref{lemma:dolb-geom-CE-std}, the standard metric on a Calabi-Eckmann manifold is geometrically Dolbeault formal if and only if the manifold is LCK \cite{KO}, whereas the standard Calabi-Eckmann manifold is geometrically Bott-Chern formal if and only if the manifold is SKT, and, more precisely,  Bismut flat \cite{WYZ}. Along the lines of Theorem \ref{thm:curv-kahl-non-neg-curv}, it would be interesting to find Hermitian connections $\tilde{\nabla}$ such that any compact Hermitian $\mathcal{R}^{\tilde{\nabla}}$-flat manifold is either geometric Dolbeault or Bott-Chern formal.
\end{rmk}

\appendix
\section{ABC-Massey products on complex solvable manifolds}\label{appendix}
In this appendix, we  write the explicit expression of a triple ABC-Massey product for each complex solvable manifold in complex dimension up to $5$ of Proposition \ref{prop:ABC-Massey-cs}, by listing structure equations, starting Bott-Chern cohomology classes and Aeppli cohomology representatives.
\vspace{0.3cm}

\begin{itemize}
    \item[($III.2$)] $d\phi^i=0$, $i=1,2$,\qquad $d\phi^3=-\phi^{12}$.
    \[
    \langle[\phi^{12}],[\phi^{\overline{12}}],[\phi^{\overline{12}}]\rangle_{ABC}=[\phi^{3\overline{123}}]_A
    \]\vspace{0.02cm}
    \item[($III.3$)] $d\phi^1=0$, \qquad $d\phi^{2}=-\phi^{12}$, \qquad $d\phi^3=\phi^{13}$.
    \[
    \langle[\phi^{12}],[\phi^{\overline{12}}],[\phi^{\overline{13}}]\rangle_{ABC}=[\phi^{2\overline{123}}]_A
    \]\vspace{0.02cm}
    \item[($IV.2$)] $d\phi^i=0$ for $i=1,2,3$, \qquad  $d\phi^4=-\phi^{23}$,\quad i.e., $M=\mathbb{T}^2\times III.2$.
    \[
    \langle[\phi^{23}],[\phi^{\overline{23}}],[\phi^{\overline{23}}]\rangle_{ABC}=[\phi^{4\overline{234}}]_A
    \]\vspace{0.02cm}
    \item[($IV.3$)] $d\phi^i=0$, $i=1,2$,\qquad  $d\phi^3=-\phi^{12}$,\qquad  $d\phi^4=-2\phi^{13}$.
    \[
    \langle[\phi^{12}],[\phi^{\overline{12}}],[\phi^{\overline{2}}]\rangle_{ABC}=[\phi^{3\overline{23}}]_A
    \]\vspace{0.02cm}
    \item[($IV.4$)] $d\phi^i=0$, $i=1,2$, \quad $d\phi^3=\phi^{23}$, \quad  $d\phi^4=-\phi^{24}$,\quad  i.e., $(M,J)\cong \mathbb{T}^2\times III.3$.
    \[
    \langle[\phi^{23}],[\phi^{\overline{23}}],[\phi^{\overline{24}}]\rangle_{ABC}=[\phi^{3\overline{234}}]_A
    \]\vspace{0.02cm}
    \item[($IV.6$)] $d\phi^1=0$, $d\phi^2=\phi^{12}$, $d\phi^3=-\phi^{13}$, $d\phi^4=-\phi^{23}$.
    \[
    \langle[\phi^{23}],[\phi^{\overline{23}}],[\phi^{\overline{23}}]\rangle_{ABC}=[\phi^{4\overline{234}}]_A
    \]\vspace{0.02cm}
    \item[($V.2$)] $d\phi^i=0$, $i=1,2,3,4$, \quad $d\phi^5=\phi^{34}$, i.e., $(M,J)=\mathbb{T}^4\times III.2$.
    \[
    \langle[\phi^{34}],[\phi^{\overline{34}}],[\phi^{\overline{34}}]\rangle_{ABC}=[\phi^{5\overline{345}}]_A
    \]\vspace{0.02cm}
    \item[($V.3$)] $d\phi^i=0$, $i=1,2,3,4$, \quad $d\phi^5=-\phi^{13}-\phi^{24}$.
    \[
    \langle[\phi^{124}],[\phi^{\overline{124}}],[\phi^{\overline{2}}]\rangle_{ABC}=[\phi^{15\overline{125}}]_A
    \]\vspace{0.02cm}
    \item[($V.4$)] $d\phi^i=0$, $i=1,2,3$, \quad $d\phi^4=-\phi^{12}$,\quad  $d\phi^5=-\phi^{13}$.
    \[
    \langle[\phi^{12}],[\phi^{\overline{12}}],[\phi^{\overline{12}}]\rangle_{ABC}=[\phi^{4\overline{124}}]_A
    \]\vspace{0.02cm}
    \item[($V.5$)] $d\phi^i=0$, $i=1,2,3$, \quad $d\phi^4=-\phi^{23}$, \quad $d\phi^5=-2\phi^{24}$, i.e., $(M,J)\cong \mathbb{T}^2\times IV.3$
    \[
    \langle[\phi^{23}],[\phi^{\overline{23}}],[\phi^{\overline{3}}]\rangle_{ABC}=[\phi^{4\overline{34}}]_A
    \]\vspace{0.02cm}
    \item[($V.6$)] $d\phi^i=0$, $i=1,2,3$, \quad $d\phi^4=-\phi^{12}$, \quad $d\phi^5=-2\phi^{14}-\phi^{23}$.
    \[
    \langle[\phi^{12}],[\phi^{\overline{12}}],[\phi^{\overline{2}}]\rangle_{ABC}=[\phi^{4\overline{24}}]_A
    \]\vspace{0.02cm}
    \item[($V.7$)] $d\phi^i=0$, $i=1,2,3$, \quad $d\phi^4=\phi^{34}$, $d\phi^5=-\phi^{35}$, i.e., $(M,J)=\mathbb{T}^2\times III.3$.
    \[
    \langle[\phi^{34}],[\phi^{\overline{34}}],[\phi^{\overline{35}}]\rangle_{ABC}=[\phi^{4\overline{345}}]_A
    \]\vspace{0.02cm}
    \item[($V.8$)] $d\phi^i=0$, $i=1,2$, \quad $d\phi^3=-\phi^{12}$, \quad$d\phi^4=-2\phi^{13}$, $d\phi^5=-2\phi^{23}$.
    \[
    \langle[\phi^{23}],[\phi^{\overline{23}}],[\phi^{\overline{2}}]\rangle_{ABC}=[\phi^{5\overline{25}}]_A
    \]\vspace{0.02cm}
    \item[($V.9$)] $d\phi^i=0$, $i=1,2$, \quad $d\phi^3=-\phi^{12}$, \quad $d\phi^4=-2\phi^{13}$, $d\phi^5=-3\phi^{14}$.
    \[
    \langle[\phi^{12}],[\phi^{\overline{12}}],[\phi^{\overline{2}}]\rangle_{ABC}=[\phi^{3\overline{23}}]_A
    \]\vspace{0.02cm}
    \item[($V.10$)] $d\phi^i=0$, $i=1,2$, \quad $d\phi^3=-\phi^{12}$, \quad $d\phi^4=-2\phi^{13}$, $d\phi^5=-3\phi^{14}-\phi^{23}$.
    \[
    \langle[\phi^{1245}],[\phi^{\overline{1245}}],[\phi^{\overline{2}}]\rangle_{ABC}=[\phi^{345\overline{2345}}]_A
    \]\vspace{0.02cm}
    \item[($V.12$)] $d\phi^i=0$, $i=1,2$, \quad $d\phi^3=\phi^{13}$, \quad $d\phi^4=\phi^{24}$, $d\phi^5=-\phi^{15}-\phi^{25}$.
    \[
    \langle[\phi^{235}],[\phi^{\overline{235}}],[\phi^{\overline{24}}]\rangle_{ABC}=[\phi^{35\overline{2345}}]_A
    \]\vspace{0.02cm}
    \item[($V.15$)] $d\phi^i=0$, $i=1,2$, \quad $d\phi^3=\phi^{23}$, \quad $d\phi^4=-\phi^{24}$, $d\phi^5=-\phi^{34}$, i.e., $(M,J)=\mathbb{T}^2\times IV.6$.
    \[
    \langle[\phi^{34}],[\phi^{\overline{34}}],[\phi^{\overline{34}}]\rangle_{ABC}=[\phi^{5\overline{345}}]_A
    \]\vspace{0.02cm}
    \item[($V.17$)] $d\phi^1=0$, $d\phi^2=\phi^{12}$, \quad $d\phi^3=\alpha\phi^{13}$, \quad $d\phi^4=\beta\phi^{14}$, $d\phi^5=-(1+\alpha+\beta)\phi^{15}$, with $\alpha\beta(1+\alpha+\beta)\neq 0$. If $\beta\neq -1$
    \[
    \langle[\phi^{124}],[\phi^{\overline{124}}],[\phi^{\overline{135}}]\rangle_{ABC}=[\phi^{24\overline{12345}}]_A.
    \]
    If $\beta=-1$, $\alpha\neq 0$,
    \[
    \langle[\phi^{1234}],[\phi^{\overline{1234}}],[\phi^{\overline{15}}]\rangle_{ABC}=[\phi^{234\overline{12345}}]_A.
    \]
\end{itemize}


\begin{thebibliography}{20}

\bibitem{AB1} L. Alessandrini, G. Bassanelli, Metric properties of manifolds bimeromorphic to compact K\"ahler
spaces, {\em J. Differ. Geom.} {\bf 37} (1993), 95--121.

\bibitem{AB2} L. Alessandrini, G. Bassanelli, The class of compact balanced manifolds is invariant under modifica-
tions, {\em Complex Analysis and Geometry} (Trento, 1993), 1--17, {\em  Lecture Notes in Pure and Appl. Math.} vol.
{\bf 173}, Dekker, NewYork (1996)

\bibitem{AZ} M. Amann, W. Ziller, Geometrically Formal Homogeneous Metrics of Positive Curvature, \emph{J. Geom. Anal.} {\bf 26} (2016), 996--1010. https://doi.org/10.1007/s12220-015-9581-y

\bibitem{AN} R.M. Arroyo, M. Nicolini, SKT structures on nilmanifolds, \emph{Math. Z.} {\bf 302} (2022), 1307--1320. https://doi.org/10.1007/s00209-022-03107-3

\bibitem{AK17} D. Angella, H. Kasuya, Bott–Chern cohomology of solvmanifolds, \emph{  Ann. Glob. Anal. Geom.}  {\bf 52} (2017), 363--411.
https://doi.org/10.1007/s10455-017-9560-6

\bibitem{AS20} D. Angella, T. Sferruzza, Geometric formalities along the Chern-Ricci flow, \emph{Complex Anal. Oper. Theory} {\bf 14} 2020, n.27. https://doi.org/10.1007/s11785-019-00971-6

\bibitem{ASTT} D. Angella, T. Suwa, N. Tardini, A. Tomassini, Note on Dolbeault cohomology and Hodge structures up to bimeromorphisms, {\em Complex Manifolds} {\bf 7} (2020), n. 1, 194--214. https://doi.org/10.1515/coma-2020-0103

\bibitem{AT13} D. Angella, A. Tomassini, On the $\del\dbar$-Lemma and Bott-Chern cohomology, \emph{ Invent. math.} {\bf 192} (2013), 71--81. https://doi.org/10.1007/s00222-012-0406-3

\bibitem{AT15} D. Angella, A. Tomassini, On Bott-Chern cohomology and formality,
\emph{J. Geom. Phys.} {\bf 93} (2015), 52--61. https://doi.org/10.1016/j.geomphys.2015.03.004

\bibitem{AG} V. Apostolov, M. Gualtieri, Generalized K\"ahler manifolds with split tangent
bundle, {\em Comm. Math. Phys.} {\bf 271} (2007) 561--575.

\bibitem{B14} C. B\"ar, Geometrically formal 4-manifolds with
nonnegative sectional curvature, \emph{Comm.  
Anal. Geom.} {\bf 23} (2015), n.3, 479--497. https://dx.doi.org/10.4310/CAG.2015.v23.n3.a2


\bibitem{BG88} C. Benson, C.S. Gordon, K\"ahler and symplectic structures on nilmanifolds, \emph{Topology} {\bf 27} (1988), n.4, 513--518.
https://doi.org/10.1016/0040-9383(88)90029-8.

\bibitem{Big69} B. Bigolin, Gruppi di Aeppli, \emph{ Ann. Scuola Norm. Sup. Pisa Cl. Sci. (3)} {\bf 23} (1969),
259--287.

\bibitem{Big70}
 B. Bigolin, Osservazioni sulla coomologia del $\del\dbar$, \emph{Ann. Scuola Norm. Sup. Pisa Cl. Sci.
(3)} {\bf 24} (1970), 571--583.

\bibitem{Bi} J.~M. Bismut, A local index theorem for non-K\"ahler manifolds, \emph{Math. Ann.} {\bf 284} (1989) 681--699.

\bibitem{Bla56} A. Blanchard, 
Sur les variétés analytiques complexes, \emph{
Annales scientifiques de l’É.N.S. 3e série} {\bf 73} (1956), n.2, 157--202. 10.24033/asens.1045

\bibitem{Bor} A. Borel, A spectral sequence for complex analytic bundles, Appendix II in  F. Hirzebruch, \emph{ Topological Methods in Algebraic Geometry}, Springer-Verlag, Berlin Heidelberg, 1995.

\bibitem{CC86} H.-D. Cao, B. Chow, Compact K\"ahler manifolds with nonnegative curvature operator, \emph{Invent. Math.} {\bf 83} (1986), 553--556. https://doi.org/10.1007/bf01394422

\bibitem{CT} A. Cattaneo, A. Tomassini,  $\del\dbar$-lemma and $p$-K\"ahler structures on families of solvmanifolds, {\em  Math. Z.} {\bf 308} (2024), n. 56. https://doi.org/10.1007/s00209-024-03612-7

\bibitem{CesTom} N. Cesarino, A. Tomassini, Aeppli-Bott-Chern Massey products on non-K\"ahler solvmanifolds, {\em Analysis and Mathematical Physics} (2026), 16:53
https://doi.org/10.1007/s13324-026-01199-2

\bibitem{CFGU} L.A. Cordero, M. Fernandez, A. Gray, L. Ugarte, Compact nilmanifolds with nilpotent complex structure: Dolbeault cohomology, \emph{Trans. Amer. Math. Soc.}
{\bf 352} (2000), n.12, 5405--5433. https://doi.org/10.1090/s0002-9947-00-02486-7

\bibitem{DGMS} P. Deligne, P. Griffiths, J. Morgan, D. Sullivan, Real homotopy theory of Kähler manifolds, \emph{Invent. Math.} {\bf 29} (1975), 245--274. https://doi.org/10.1007/BF01389853



\bibitem{FT} A. Fino, A. Tomassini, Non K\" ahler solvmanifolds with generalized K\"ahler
structure, {\em J. Sympl.
Geom} {\bf 7} (2009), 1--14.

\bibitem{FV} A. Fino, L. Vezzoni, On the existence of balanced and SKT metrics on nilmanifolds, \emph{Proc. Amer. Math. Soci.} {\bf 144} (2016), n.6.
http://dx.doi.org/10.1090/proc/12954.

\bibitem{Fro55} A. Fr\"olicher, Relations between the cohomology groups of Dolbeault and topological invariants, \emph{ Proc. Nat. Acad. Sci. USA} {\bf 41} (1955) n.9, 641--644. https://doi.org/10.1073/pnas.41.9.641 

\bibitem{Ga2} P. Gauduchon, Hermitian connections and Dirac operators, \emph{Boll. Un. Mat. It.} {\bf 11-B} (1997)
Suppl. Fasc. 257--288.

\bibitem{GM75} S. Gallot, D. Meyer,
Opérateur de courbure et laplacien des formes différentielles d’une variété riemannienne (French), \emph{
J. Math. Pur. Appl., IX. Sér.} {\bf 54} (1975), 259--284.

\bibitem{GHR} S.~J. Gates, C.~M. Hull, M. Ro\v cek, Twisted multiplets and new supersymmetric nonlinear sigma models,
{\em Nuc. Phys. B} {\bf 248} (1984) 157--186.

\bibitem{GriHar} P. Griffiths, J. Harris, \emph{Principles of algebraic geometry}, John Wiley and Sons, INC,
Wiley Classics Library, Harvard, 1994.

\bibitem{Gu} M. Gualtieri, Generalized complex geometry, {\em Annals of Math.} {\bf 174} (2011), 75--123. 
https://doi.org/10.4007/annals.2011.174.1.3



\bibitem{Has06} K. Hasegawa, A note on compact solvmanifolds with K\"ahler structures, \emph{Osaka J. Math.} {\bf 43} (2006), n.1, 131--135.

\bibitem{Has89} K. Hasegawa, Minimal Models of Nilmanifolds, \emph{Proc. Amer. Math. Soc.} {\bf 106} 1989, n.1, 65--71.  https://doi.org/10.1090/S0002-9939-1989-0946638-X

\bibitem{Hi2} N.~J. Hitchin, Instantons and generalized K\"ahler geometry, {\em Comm. Math. Phys.} {\bf 265} (2006) 131--164.

\bibitem{IK} H. Ishida, H. Kasuya, Transverse K\"ahler structures on central foliations
of complex manifolds, {\em Ann. Mat. Pura Appl. (4)} {\bf 198} (2019), no. 1, 61--81.

\bibitem{Huy00} D. Huybrechts, Products of harmonic forms and rational curves, \emph{Doc. Math.} {\bf 6} (2001), 227--239. https://doi.org/10.4171/dm/102

\bibitem{IO23} N. Istrati, A. Otiman, Bott-Chern cohomology of compact Vaisman manifolds, {\em Trans. Amer. Math.  Soc.} {\bf 376} (2023), 3919--3936.


\bibitem{KO} Y. Kamishima, L. Ornea, Geometric flow on compact locally conformally
K¨ahler manifolds, \emph{Tohoku Math. J.}, {\bf 57} (2005), n.5, 201--221.

\bibitem{Kasuya-Vaisman-metrics} H. Kasuya, Vaisman metrics on solvmanifolds and Oeljeklaus-Toma manifolds, {\em Bull. Lond. Math. Soc.} {\bf 45} (2013), no. 1, 15--26. 

\bibitem{Kasuya-Kokyu} H. Kasuya, Geometrical formality of solvmanifolds and solvable
Lie type geometries,  {\em RIMS Kôkyûroku Bessatsu}
{\bf B39} (2013), 021--033.


\bibitem{Kasuya-Hodge-sym-dec} H. Kasuya,
Hodge symmetry and decomposition on non-Kähler solvmanifolds,
{\em J. Geom. Phys.},
{\bf 76}
(2014),
61--65.
https://doi.org/10.1016/j.geomphys.2013.10.012

\bibitem{Kasuya-Remarks-Dolbeault} H. Kasuya, Remarks on Dolbeault cohomology of Oeljeklaus-Toma manifolds and Hodge theory, {\em Proc. Amer. Math. Soc.} {\bf 149} (2021), no. 7,
3129--3137.

\bibitem{Kas} H. Kasuya, Techniques of computations of Dolbeault cohomology
of solvmanifolds, \emph{Math. Z.} {\bf 273} (2013), 437--447. https://doi.org/10.1007/s00209-012-1013-0


\bibitem{KasSte} H. Kasuya, J. Stelzig, Complex non-K\"ahler manifolds that are cohomologically close to, or far from, being K\"ahler, {\em Math. Res. Lett.} {\bf 32} (2025), n. 4, 1133--1164.
https://dx.doi.org/10.4310/MRL.251124115522

\bibitem{Kob63} S. Kobayashi, Topology of positively pinched Kaehler manifolds, \emph{Tohoku Math. J. (2)} {\bf 15} (1963), n.2, 121--139. https://doi.org/10.2748/tmj/1178243839 



\bibitem{Kot01} D. Kotschick, On products of harmonic forms, \emph{Duke Math. J.} {\bf 107} (2001), n.3. https://doi.org/10.1215/s0012-7094-01-10734-5 

\bibitem{Kot1} D. Kotschick, Geometric formality and non-negative scalar curvature, \emph{Pure Appl. Math. Q.}
{\bf 13} (2017), n.3, 437--451. https://doi.org/10.4310/pamq.2017.v13.n3.a3 

\bibitem{KT1} D. Kotschick, S. Terzić, Geometric formality of homogeneous spaces and biquotients, \emph{ Pacific J. Math.}
{\bf 249} (2011), n.1, 157--176. https://doi.org/10.2140/pjm.2011.249.157 

\bibitem{KT2} D. Kotschick, S. Terzić, On formality of generalized symmetric spaces,  {\em Math. Proc. Cambridge Philos. Soc.} {\bf 134} (2003), n.3, 491--505. https://doi.org/10.1017/s0305004102006540 

\bibitem{Meer} L. Meersseman, A new geometric construction of compact complex
manifolds in any dimension, {\em Math. Annalen}, {\bf 317} (2000), 79--115.

\bibitem{Mer} S. Merkulov, Strong homotopy algebras of a K\"ahler manifold, \emph{Int. Math. Res. Notices} {\bf 3} (1999), 153--164. https://doi.org/10.1155/S1073792899000070

\bibitem{Mi} M.L. Michelsohn, On the existence of special metrics in complex geometry, {\em Acta Math.} {\bf 149} (1982), 261--295. 

\bibitem{MS24} A. Milivojevic, J. Stelzig, Bigraded notions of formality and Aeppli-Bott-Chern Massey products, \emph{Comm. Anal. Geom.} {\bf 32} (2024), n.10, 2901--2933.
https://dx.doi.org/10.4310/CAG.241231035705.

\bibitem{Nagy} P.-A. Nagy, On length and product of harmonic forms in K\"ahler geometry,
\emph{Math. Z.} {\bf 254} (2006), 199--218. https://doi.org/10.1007/s00209-006-0942-x

\bibitem{Nak75} I. Nakamura, On complex parallelisable manifolds and their small deformations, \emph{J. Diff. Geom.} {\bf 10} (1975) n.1, 85--112. https://doi.org/10.4310/jdg/1214432677 


\bibitem{NT78} J. Neisendorfer, L. Taylor, Dolbeault homotopy theory, \emph{Trans. Amer. Math. Soc.} {\bf 245} (1978), 183--210. https://doi.org/10.2307/1998862 

\bibitem{OrnPil} L. Ornea, M. Pilca, Remarks on the product of harmonic forms, \emph{Pacific J. Math.}
{\bf 250} (2011), n.2, 353--364. https://doi.org/10.2140/pjm.2011.250.353 

\bibitem{Pit} H.V. Pittie, The Dolbeault-cohomology ring of a compact, even-dimensional lie group, {\em Proc. Indian Acad. Sci. (Math. Sci.)} {\bf 98} (1988), 117--152. https://doi.org/10.1007/BF02863632

\bibitem{PSZ24} G. Placini, J. Stelzig, L. Zoller, Nontrivial Massey products on compact K\"ahler manifolds, {\tt 	arXiv:2404.09867} 
https://doi.org/10.48550/arXiv.2404.09867

\bibitem{Pol} A. Polishchuk, Homological mirror symmetry with higher products, {\tt 	arXiv:math/9901025}  https://doi.org/10.48550/arXiv.math/9901025

\bibitem{Rog} V. Rogov, Non-algebraic deformations of flat K\"ahler manifolds, \emph{Math. Res. Lett.} {\bf 29} (2022), n.4, 1229--1250.
https://dx.doi.org/10.4310/MRL.2022.v29.n4.a12

\bibitem{Rub} L. Rubini,
Some computations on trivial canonical-bundle solvmanifolds,
{\em J. Geom. Phys.} {\bf 216} (2025), 105586. https://doi.org/10.1016/j.geomphys.2025.105586.

\bibitem{Sam} H. Samelson, A class of complex analytic manifolds, \emph{Portugaliae Mathematica} {\bf  12} (1953), 129--132.

\bibitem{Sch01}  M. Schweitzer, Autour de la cohomologie de Bott-Chern, {\tt 	arXiv:0709.3528} https://doi.org/10.48550/arXiv.0709.3528

\bibitem{Sfe23} T. Sferruzza, Formality of special complex manifolds: deformations and $\del\dbar$-lemma, (PhD thesis) https://hdl.handle.net/1889/5386

\bibitem{ST24} T. Sferruzza, A. Tomassini, Bott-Chern formality and Massey products of strong K\"ahler with torsion and K\"ahler solvmanifolds, \emph{J. Geom. Anal.} {\bf 34} (2024) n. 348. https://doi.org/10.1007/s12220-024-01764-w

\bibitem{SfeTom22} T. Sferruzza, A. Tomassini, Dolbeault and Bott-Chern formalities: deformations and $\del\dbar$-lemma, \emph{J. Geom. Phys.} {\bf 175} (2022), 104470. https://doi.org/10.1016/j.geomphys.2022.104470 

\bibitem{SV} T. Sferruzza, M. Verbitsky, Dolbeault formality for complex nilmanifolds, {\tt 	arXiv:2602.13426} 
https://doi.org/10.48550/arXiv.2602.13426

\bibitem{Ste21} J. Stelzig, On the structure of double complexes, \emph{J. London Math. Soc. (2)} {\bf 104} (2021) 956–-988. doi:10.1112/jlms.12453. https://doi.org/10.1112/jlms.12453 


\bibitem{Ste22} J. Stelzig, The Double Complex of a Blow-up, \emph{Int. Math. Res. Not.}, {\bf 2021} (2021), n.14, 10731--10744. https://doi.org/10.1093/imrn/rnz139.

\bibitem{Str} A. Strominger, Superstrings with torsion, {\em Nuclear Phys. B}
{\bf 274} (1986) 253--284.

\bibitem{Sull73} D. Sullivan, Diﬀerential Forms and the Topology of Manifolds, in {\em Manifolds}, Tokyo
1973, ed. A. Hattori, University of Tokyo Press, 1975.


\bibitem{Tan94} D. Tanré, Modèle de Dolbeault et fibré holomorphe, {\em J. Pure Appl. Algebra} {\bf 91} (1994), 333--345. https://doi.org/10.1016/0022-4049(94)90149-x 

\bibitem{TT17} N. Tardini, A. Tomassini, On geometric Bott-Chern formality and deformations, \emph{Annali di Matematica} {\bf 196} (2017), 349--362. https://doi.org/10.1007/s10231-016-0575-6

\bibitem{TomTor14} A. Tomassini, S. Torelli, On Dolbeault formality and small deformations,  \emph{Int. J. Math.} {\bf 25} 2014, n.11, 1450111. https://doi.org/10.1142/S0129167X14501110

\bibitem{Wan54} H.-C. Wang, Complex parallizable manifolds {\em Proc. Amer. Math. Soc.} {\bf 5}  (1954), 771--776. https://doi.org/10.1090/S0002-9939-1954-0074064-3

\bibitem{WYZ} Q. Wang, B. Yang, F. Zheng, On Bismut flat manifolds, \emph{Trans. Amer. Math. Soc.} {\bf 373} (2020), 5747--5772. https://doi.org/10.1090/tran/8083 

\bibitem{YanYan20} S. Yang, X. Yang, Bott-Chern blow-up formulae and the bimeromorphic invariance of the $\del\dbar$-lemma for threefolds, \emph{Trans. Amer. Math. Soc. } {\bf 373} (2020), n.12, 8885--8909. https://doi.org/10.1090/tran/8213 

\end{thebibliography}
\end{document}